\documentclass{amsart}
\usepackage{amssymb, amsmath, mathrsfs, verbatim, url}
\usepackage[all,cmtip]{xy}
\usepackage{mathabx}

\theoremstyle{plain}
\newtheorem{theorem}{Theorem}[section]
\newtheorem{lemma}[theorem]{Lemma}
\newtheorem{proposition}[theorem]{Proposition}
\newtheorem{corollary}[theorem]{Corollary}

\theoremstyle{definition}
\newtheorem{definition}[theorem]{Definition}
\newtheorem{question}[theorem]{Question}

\newcommand{\re}{\upharpoonright}

\newcommand{\id}{\mathsf{id}}

\newcommand{\bD}{\mathbf{\Delta}}
\newcommand{\bG}{\mathbf{\Gamma}}
\newcommand{\bGc}{\mathbf{\check{\Gamma}}}
\newcommand{\bL}{\mathbf{\Lambda}}
\newcommand{\bLc}{\mathbf{\check{\Lambda}}}
\newcommand{\ZFC}{\mathsf{ZFC}}
\newcommand{\ZF}{\mathsf{ZF}}
\newcommand{\DC}{\mathsf{DC}}
\newcommand{\NSD}{\mathsf{NSD}}
\newcommand{\PU}{\mathsf{PU}}

\newcommand{\AD}{\mathsf{AD}}
\newcommand{\BP}{\mathsf{BP}}

\newcommand{\SD}{\mathsf{SD}}
\newcommand{\W}{\mathsf{W}}

\newcommand{\Aa}{\mathcal{A}}
\newcommand{\BB}{\mathcal{B}}
\newcommand{\CC}{\mathcal{C}}

\newcommand{\FF}{\mathcal{F}}

\newcommand{\HH}{\mathcal{H}}

\newcommand{\UU}{\mathcal{U}}

\newcommand{\PP}{\mathcal{P}}

\newcommand{\RRR}{\mathbb{R}}

\begin{document}

\title[Every zero-dimensional homogeneous space is strongly homogeneous]{Every zero-dimensional homogeneous space is strongly homogeneous under determinacy}

\author{Rapha\"el Carroy}
\address{Department of Mathematics ``Giuseppe Peano''
\newline\indent Palazzo Campana, University of Turin
\newline\indent Via Carlo Alberto 10
\newline\indent 10123 Turin, Italy}
\email{raphael.carroy@unito.it}
\urladdr{http://www.logique.jussieu.fr/\~{}carroy/indexeng.html}

\author{Andrea Medini}
\address{Kurt G\"odel Research Center for Mathematical Logic
\newline\indent Institute of Mathematics, University of Vienna
\newline\indent Augasse 2-6, UZA 1 - Building 2
\newline\indent 1090 Vienna, Austria}
\urladdr{http://www.logic.univie.ac.at/\~{}medinia2/}

\author{Sandra M\"uller}
\address{Kurt G\"odel Research Center for Mathematical Logic
\newline\indent Institute of Mathematics, University of Vienna
\newline\indent Augasse 2-6, UZA 1 - Building 2
\newline\indent 1090 Vienna, Austria}
\email{mueller.sandra@univie.ac.at}
\urladdr{https://muellersandra.github.io}

\subjclass[2010]{54H05, 03E15, 03E60.}

\keywords{Homogeneous, strongly homogeneous, h-homogeneous, zero-dimensional, determinacy, Wadge theory, Hausdorff operation, $\omega$-ary Boolean operation.}

\thanks{The first-listed author acknowledges the support of the FWF grant P 28153-N35. The second-listed author acknowledges the support of the FWF grant P 30823-N35. The third-listed author (formerly known as Sandra Uhlenbrock) acknowledges the support of the FWF grant P 28157-N35. The authors are grateful to Alessandro Andretta for valuable bibliographical help, and to Alain Louveau for allowing them to use his unpublished book \cite{louveaub}.}

\date{February 26, 2020}

\begin{abstract}
All spaces are assumed to be separable and metrizable. We show that, assuming the Axiom of Determinacy, every zero-dimensional homogeneous space is strongly homogeneous (that is, all its non-empty clopen subspaces are homeomorphic), with the trivial exception of locally compact spaces. In fact, we obtain a more general result on the uniqueness of zero-dimensional homogeneous spaces which generate a given Wadge class. This extends work of van Engelen (who obtained the corresponding results for Borel spaces), complements a result of van Douwen, and gives partial answers to questions of Terada and Medvedev.
\end{abstract}

\maketitle

\section{Introduction}

Throughout this article, unless we specify otherwise, we will be working in the theory $\ZF+\DC$, that is, the usual axioms of Zermelo-Fraenkel (without the Axiom of Choice) plus the principle of Dependent Choices (see Section 2 for more details). By \emph{space} we will always mean separable metrizable topological space, unless we specify otherwise. A space $X$ is \emph{homogeneous} if for every $x,y\in X$ there exists a homeomorphism $h:X\longrightarrow X$ such that $h(x)=y$. For example, using translations, it is easy to see that every topological group is homogeneous (as \cite[Corollary 3.6.6]{vanengelent} shows, the converse is not true, not even for zero-dimensional Borel spaces). Homogeneity is a classical notion in topology, which has been studied in depth (see for example \cite{arkhangelskiivanmill}). In particular, in his remarkable doctoral thesis \cite{vanengelent} (see also \cite{vanengelenpr} and \cite{vanengelentr}), Fons van Engelen gave a complete classification of the homogeneous zero-dimensional Borel spaces. In fact, as we will make more precise, this article is inspired by his work and relies heavily on some of his techniques.

A space $X$ is \emph{strongly homogeneous} (or \emph{h-homogeneous}) if every non-empty clopen subspace of $X$ is homeomorphic to $X$. This notion has been studied by several authors, both ``instrumentally'' and for its own sake (see the list of references in \cite{medinia}). It is well-known that every zero-dimensional strongly homogeneous space is homogeneous (see for example \cite[1.9.1]{vanengelent} or \cite[Proposition 3.32]{medinit}). Our main result shows that, under the Axiom of Determinacy (briefly, $\AD$) the converse also holds (with the trivial exception of locally compact spaces, see Proposition \ref{locallycompact}). For the proof, see Corollary \ref{main}.

\begin{theorem}\label{mainintro}
Assume $\AD$. If $X$ is a zero-dimensional homogeneous space that is not locally compact then $X$ is strongly homogeneous.
\end{theorem}

The above theorem follows from a uniqueness result about zero-dimensional homogeneous spaces, namely Theorem \ref{uniqueness}, which is of independent interest.\footnote{\,In fact, the only consequences of $\AD$ that Theorem \ref{uniqueness} (hence Theorem \ref{mainintro}) requires are $\BP$ (see Section 2) and Lemma \ref{wadgelemma}.} This theorem essentially states that, for every sufficiently high level of complexity $\bG$, there are at most two homogeneous zero-dimensional spaces of complexity exactly $\bG$ (one meager and one Baire).

Our fundamental tool will be Wadge theory, which was founded by William Wadge in his doctoral thesis \cite{wadget} (see also \cite{wadgea}), and has become a classical topic in descriptive set theory. We believe that \cite[Theorem 2.4]{vanengelenmillersteel} and our results are the only applications to topology of an analysis of the full (as opposed to just Borel) Wadge hierachy.\footnote{\,While all of these results ultimately rely on \cite[Theorem 2]{steela}, the techniques used in the proof of \cite[Theorem 2.4]{vanengelenmillersteel} are very different from those used here.} In fact, most of this article (Sections 3 to 13) is purely Wadge-theoretic in character. The ultimate goal of the Wadge-theoretic portion of the paper is to show that good Wadge classes are closed under intersection with $\mathbf{\Pi}^0_2$ sets (see Section 12), hence they are reasonably closed (see Section 13). Homogeneity comes into play in Section 14, where we show that $[X]$ is a good Wadge class whenever $X$ is a homogeneous space of sufficiently high complexity. This will allow us to use a theorem of Steel from \cite{steela}, which will in turn yield the uniqueness result mentioned above (see Section 15). The following diagram summarizes the structure of the proof of our main result. In the preceding sections, the necessary tools are developed. More specifically, Section 4 is devoted to the analysis of the selfdual Wadge classes, Sections 5 to 8 develop the machinery of relativization through Hausdorff operations, and Sections 9 to 11 develop the notions of level and expansion. While none of the ideas contained in these preliminary sections are new (except, to the best of our knowledge, Theorem \ref{closureGdelta}), satisfactory references are hard to come by, especially for the required level of generality. For this reason, we will always restate the needed results (and sometimes give full proofs).

\smallskip
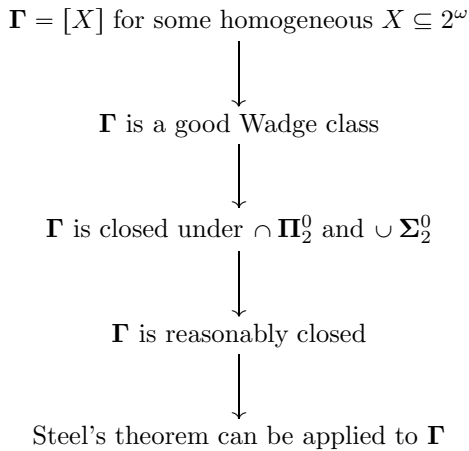
\begin{center}
$
\xymatrix{
\text{$\bG=[X]$ for some homogeneous $X\subseteq 2^\omega$}\ar[d]\\
\text{$\bG$ is a good Wadge class}\ar[d]\\
\text{$\bG$ is closed under $\cap\,\mathbf{\Pi}^0_2$ and $\cup\,\mathbf{\Sigma}^0_2$}\ar[d]\\
\text{$\bG$ is reasonably closed}\ar[d]\\
\text{Steel's theorem can be applied to $\bG$}
}
$
\end{center}
\smallskip

The application of Wadge theory to the study of homogeneous spaces was pioneered by van Engelen in \cite{vanengelent}, where he obtained the classification mentioned above. As a corollary (namely, \cite[Corollary 4.4.6]{vanengelent}), he obtained the Borel version of Theorem \ref{mainintro}. The reason why his results are limited to Borel spaces is that they are all based on the fine analysis of the Borel Wadge classes given by Louveau in \cite{louveaua}. Fully extending this analysis beyond the Borel realm appears to be a very hard problem (although partial results have been obtained in \cite{fournier}). Here, we will follow a different strategy, and we will ``substitute'' facts from \cite{louveaua} about Borel Wadge classes with more general results about arbitrary Wadge classes (under $\AD$). Furthermore, since most of the literature on Wadge theory only deals with $\omega^\omega$ as the ambient space, while Steel's theorem is stated for $2^\omega$, we decided to work in the context of arbitrary zero-dimensional uncountable Polish spaces. With regard to these issues, Louveau's book \cite{louveaub} and Van Wesep's results on Hausdorff operations from \cite{vanwesept} were crucial. For other applications of Wadge theory to topology, see the characterizations of Borel filters and semifilters given respectively in \cite{vanengelena} and \cite{medinis}.

At this point, it is natural to ask whether assuming $\AD$ is really necessary in the above results. As the following theorem shows, the answer is ``yes''. This result was essentially proved in \cite{vandouwen}, but our exposition is based on \cite[Theorem 5.1]{vanmill}. Following \cite{vanmill}, we will say that $X\subseteq\RRR$ is a \emph{bi-Bernstein set}\footnote{\,These sets are commonly referred to simply as \emph{Bernstein sets}.} if $K\cap X\neq\varnothing$ and $K\cap (\RRR\setminus X)\neq\varnothing$ for every $K\subseteq\RRR$ that is homeomorphic to $2^\omega$.

\begin{theorem}[van Douwen]\label{vandouwentheorem}
There exists a $\ZFC$ example $X$ of a homogeneous zero-dimensional space that is not locally compact and not strongly homogeneous.
\end{theorem}
\begin{proof}
Let $X$ be the space given by \cite[proof of Theorem 5.1]{vanmill}. Notice that $X$ is homogeneous because $X$ is a subgroup of $\RRR$. Furthermore, $X$ is a bi-Bernstein set by \cite[Proposition 4.5]{vanmill}. It follows that both $X$ and $\RRR\setminus X$ are dense in $\RRR$. In particular, $X$ is zero-dimensional and not locally compact.

Given any Borel subset $B$ of $X$, pick a Borel subset $A$ of $\RRR$ such that $A\cap X=B$, then define $\overline{\mu}(B)=\mu(A)$, where $\mu$ denotes the Lebesgue measure on $\RRR$. Using the fact that $X$ is bi-Bernstein, it is easy to check that $\overline{\mu}$ is a well-defined measure on the Borel subsets of $X$. The crucial property of $\overline{\mu}$, as given by the statement of \cite[Theorem 5.1]{vanmill}, is that if $B$ and $C$ are homeomorphic Borel subspaces of $X$, then $\overline{\mu}(B)=\overline{\mu}(C)$. 

Now pick $a,b,c\in\RRR\setminus X$ such that $a<b<c$. Observe that $U=(a,b)\cap X$ and $V=(a,c)\cap X$ are non-empty clopen subsets of $X$. Furthermore, it is clear from the definition of $\overline{\mu}$ that $\overline{\mu}(U)=b-a\neq c-a=\overline{\mu}(V)$. Therefore $U$ and $V$ are not homeomorphic, which concludes the proof.
\end{proof}

However, we do not know the answer to the following question. Recall that, when $\bG=\mathbf{\Sigma}^1_n$ or $\bG=\mathbf{\Pi}^1_n$ for some $n\geq 1$, a space is $\bG$ if it is homeomorphic to a $\bG$ subspace of some Polish space (see \cite[Section 4]{medinizdomskyy} for a more detailed treatment).
\begin{question}
Assuming $\mathsf{V}=\mathsf{L}$, is it possible to construct a zero-dimensional $\mathbf{\Pi}^1_1$ or $\mathbf{\Sigma}^1_1$ space that is homogeneous, not locally compact, and not strongly homogeneous?
\end{question}
The above question is natural because there are many examples of properties (such as the perfect set property\footnote{\,To see that every space has the perfect set property under $\AD$, proceed as in \cite[Section 21.A]{kechris}. For the counterexample under $\mathsf{V}=\mathsf{L}$, see \cite[Theorem 13.12]{kanamori}.}) that are known to hold for all spaces under $\AD$, for which definable counterexamples can be constructed under $\mathsf{V}=\mathsf{L}$. Notice that $\mathbf{\Pi}^1_1$ and $\mathbf{\Sigma}^1_1$ are optimal by \cite[Corollary 4.4.6]{vanengelent}. For other relevant examples, see \cite[Theorem 2.6]{vanengelenmillersteel}, \cite{miller}, and \cite{vidnyanszky}.

Finally, we mention three applications of Theorem \ref{mainintro}. The first is that Theorem \ref{vandouwentheorem} cannot be proved without using the Axiom of Choice (more precisely, it cannot be proved in $\ZF+\DC$ alone).
The second concerns the following question from \cite[Section 3]{terada} (see \cite[Section 3]{medinia} and \cite[Section 5]{medinivanmillzdomskyy} for more on this topic).
\begin{question}[Terada]\label{teradaquestion}
Is $X^\omega$ strongly homogeneous for every zero-dimensional first-countable space $X$?
\end{question}
Since Lawrence showed that $X^\omega$ is homogeneous for every zero-dimensional space $X$ (see \cite{lawrence}, or \cite{dowpearl} for a more general result), it follows from Theorem \ref{mainintro} that the answer to Question \ref{teradaquestion} in the separable metrizable realm is ``yes'' under $\AD$.\footnote{\,Notice that if a zero-dimensional infinite power is locally compact then it is compact, hence either it has size $1$, or it is homeomorphic to $2^\omega$ by \cite[Theorem 7.4]{kechris}.} As the third application, we obtain that the answer to the following question (see \cite[Question 1]{medvedev}) is also ``yes'' under $\AD$.\footnote{\,Recall that locally compact spaces are Baire.}
\begin{question}[Medvedev]
Is every zero-dimensional meager homogeneous space strongly homogeneous?
\end{question}

\section{Preliminaries and notation}

Let $Z$ be a set, and let $\bG\subseteq\PP(Z)$. Define $\bGc=\{Z\setminus A:A\in\bG\}$. We will say that $\bG$ is \emph{selfdual} if $\bG=\bGc$. Also define $\Delta(\bG)=\bG\cap\bGc$. Given a function $f:Z\longrightarrow W$, $A\subseteq Z$, and $B\subseteq W$, we will use the notation $f[A]=\{f(x):x\in A\}$ and $f^{-1}[B]=\{x\in Z:f(x)\in B\}$.

\begin{definition}[Wadge]
Let $Z$ be a space, and let $A,B\subseteq Z$. We will write $A\leq B$ if there exists a continuous function $f:Z\longrightarrow Z$ such that $A=f^{-1}[B]$.\footnote{\,Wadge-reduction is usually denoted by $\leq_\mathsf{W}$, which allows to distinguish it from other types of reduction (such as Lipschitz-reduction). Since we will not consider any other type of reduction in this article, we decided to simplify the notation.} In this case, we will say that $A$ is \emph{Wadge-reducible} to $B$, and that $f$ \emph{witnesses} the reduction. We will write $A<B$ if $A\leq B$ and $B\not\leq A$. We will write $A\equiv B$ if $A\leq B$ and $B\leq A$.
\end{definition}

\begin{definition}[Wadge]\label{wadgeclassdefinition}
Let $Z$ be a space. Given $A\subseteq Z$, define
$$
[A]=\{B\subseteq Z:B\leq A\}.\footnote{\,Usually, this denotes the \emph{Wadge degree} of $A$, that is $\{B\subseteq Z:B\equiv A\}$. Our notation follows \cite{vanengelenpr}, \cite{vanengelent}, \cite{vanengelena}, and \cite{louveaua}, since these were among our main inspirations and sources. Even the book \cite{louveaub} uses the similar (but slightly more cumbersome) notation $\bG[A]$.}
$$
We will say that $\bG\subseteq\PP(Z)$ is a \emph{Wadge class} if there exists $A\subseteq Z$ such that $\bG=[A]$. We will say that $\bG\subseteq\PP(Z)$ is \emph{continuously closed} if $[A]\subseteq\bG$ for every $A\in\bG$.
\end{definition}

Both of the above definitions depend of course on the space $Z$. Often, for the sake of clarity, we will specify what the ambient space is by saying, for example, that ``$A\leq B$ in $Z$'' or ``$\bG$ is a Wadge class in $Z$''. We will say that $A\subseteq Z$ is \emph{selfdual} if $A\equiv Z\setminus A$ in $Z$. It is easy to check that $A$ is selfdual iff $[A]$ is selfdual. Given a space $Z$, we will also use the following shorthand notation:
\begin{itemize}
\item $\SD(Z)=\{\bG:\bG\textrm{ is a selfdual Wadge class in }Z\}$,
\item $\NSD(Z)=\{\bG:\bG\textrm{ is a non-selfdual Wadge class in }Z\}$.
\end{itemize}

Our reference for descriptive set theory is \cite{kechris}. In particular, we assume familiarity with the basic theory of Borel sets and Polish spaces, and use the same notation as in \cite[Section 11.B]{kechris}. For example, given a space $Z$, we use $\mathbf{\Sigma}^0_1(Z)$, $\mathbf{\Pi}^0_1(Z)$, and $\mathbf{\Delta}^0_1(Z)$ to denote the collection of all open, closed, and clopen subsets of $Z$ respectively. Our reference for other set-theoretic notions is \cite{jech}.

The classes defined below constitute the so-called \emph{difference hierarchy} (or \emph{small Borel sets}). For a detailed treatment, see \cite[Section 22.E]{kechris} or \cite[Chapter 3]{vanengelent}. Here, we will only mention that the $\mathsf{D}_\eta(\mathbf{\Sigma}^0_\xi(Z))$ are among the simplest concrete examples of Wadge classes (see Proposition \ref{expansiondifferences} and Corollary \ref{expansionhausdorffnonselfdual}).

\begin{definition}[Kuratowski]
Let $Z$ be a space, let $1\leq\eta<\omega_1$ and $1\leq\xi<\omega_1$. Given a sequence of sets $(A_\mu:\mu<\eta)$, define
$$
\left.
\begin{array}{lcl}
& & \mathsf{D}_\eta(A_\mu:\mu<\eta)= \left\{
\begin{array}{ll}
\bigcup\{A_\mu\setminus\bigcup_{\zeta<\mu}A_\zeta:\mu<\eta\text{ and }\mu\text{ is odd}\} & \text{if }\eta\text{ is even,}\\
\bigcup\{A_\mu\setminus\bigcup_{\zeta<\mu}A_\zeta:\mu<\eta\text{ and }\mu\text{ is even}\} & \text{if }\eta\text{ is odd.}
\end{array}
\right.
\end{array}
\right.
$$
Define $A\in\mathsf{D}_\eta(\mathbf{\Sigma}^0_\xi(Z))$ if there exist $A_\mu\in\mathbf{\Sigma}^0_\xi(Z)$ for $\mu<\eta$ such that $A=\mathsf{D}_\eta(A_\mu:\mu<\eta)$.\footnote{\,Notice that the definition of $\mathsf{D}_\eta(\mathbf{\Sigma}^0_\xi(Z))$ does not change if one adds the requirement that $(A_\mu:\mu<\eta)$ is $\subseteq$-increasing.}
\end{definition}

For an introduction to the topic of games, we refer the reader to \cite[Section 20]{kechris}. Here, we only want to give the precise definition of determinacy. A \emph{play} of the game $G(\omega,X)$ is decribed by the diagram
\begin{center}
\begin{tabular}{cccccl}
I & $a_0$ & & $a_2$ & & $\cdots$\\
II & & $a_1$ & & $a_3$ & $\cdots$,
\end{tabular}
\end{center}
in which $a_n\in\omega$ for every $n\in\omega$ and $X\subseteq \omega^\omega$ is called the \emph{payoff set}. We will say that Player I \emph{wins} this play of the game $G(\omega,X)$ if $(a_0,a_1,\ldots)\in X$. Player II \emph{wins} if Player I does not win.

A \emph{strategy} for a player is a function $\sigma:\omega^{<\omega}\longrightarrow\omega$. We will say that $\sigma$ is a \emph{winning strategy} for Player I if setting $a_{2n}=\sigma(a_1,a_3,\ldots,a_{2n-1})$ for each $n$ makes Player I win for every $(a_1,a_3,\ldots)\in\omega^\omega$. A winning strategy for Player II is defined similarly. We will say that the game $G(\omega,X)$ (or simply the set $X$) is \emph{determined} if (exactly) one of the players has a winning strategy. The \emph{Axiom of Determinacy} (briefly, $\AD$) states that every $X\subseteq\omega^\omega$ is determined.\footnote{\,Quite amusingly, Van Wesep referred to $\AD$ as a ``frankly heretical postulate'' (see \cite[page 64]{vanwesept}), and Steel deemed it ``probably false'' (see \cite[page 63]{steelt}).} We will denote by $\BP$ the axiom stating that for every Polish space $Z$, every subset of $Z$ has the Baire property. Using the arguments in \cite[Section 21.C]{kechris}, it can be shown that $\AD$ implies $\BP$.

It is well-known that $\AD$ is incompatible with the Axiom of Choice (see \cite[Lemma 33.1]{jech}). This is the reason why, throughout this article, we will be working in $\ZF+\DC$.\footnote{\,The consistency of $\ZF+\DC+\AD$ can be obtained under suitable large cardinal assumptions (see \cite[Proposition 11.13]{kanamori} and \cite{neeman}).} Recall that the \emph{principle of Dependent Choices} (briefly, $\DC$) states that if $R$ is a binary relation on a non-empty set $A$ such that for every $a\in A$ there exists $b\in A$ such that $(b,a)\in R$, then there exists a sequence $(a_0, a_1,\ldots)\in A^\omega$ such that $(a_{n+1},a_n)\in R$ for every $n\in\omega$. An equivalent formulation of $\DC$ is that a relation $R$ on a set $A$ is well-founded iff there exists no sequence $(a_0, a_1,\ldots)\in A^\omega$ such that $(a_{n+1},a_n)\in R$ for every $n\in\omega$ (see \cite[Lemma 5.5.ii]{jech}). Furthermore, $\DC$ implies the Countable Axiom of Choice (see \cite[Exercise 5.7]{jech}). To the reader who is unsettled by the lack of the full Axiom of Choice, we recommend \cite{howardrubin}.

We conclude this section with some miscellaneous topological definitions and results. We will write $X\approx Y$ to mean that the spaces $X$ and $Y$ are homeomorphic. A subset of a space is \emph{clopen} if it is closed and open. A space is \emph{zero-dimensional} if it is non-empty and it has a base consisting of clopen sets.\footnote{\,The empty space has dimension $-1$.} Given a function $s:F\longrightarrow 2$, where $F\subseteq\omega$ is finite, we will use the notation $[s]=\{z\in 2^\omega:s\subseteq z\}$. A space is \emph{crowded} if it is non-empty and it has no isolated points. A space $X$ is \emph{Baire} if every non-empty open subset of $X$ is non-meager in $X$. A space $X$ is \emph{meager} if $X$ is a meager subset of $X$. Proposition \ref{meagerorbaire} is a particular case of \cite[Lemma 3.1]{fitzpatrickzhou} (see also \cite[1.12.1]{vanengelent}). Proposition \ref{locallycompact} is the reason why we refer to locally compact spaces as the ``trivial exceptions''. Theorem \ref{pibase} is a special case of \cite[Theorem 2.4]{terada} (see also \cite[Theorem 2 and Appendix A]{medinia} or \cite[Theorem 3.2 and Appendix B]{medinit}).

\begin{proposition}[Fitzpatrick, Zhou]\label{meagerorbaire}
Let $X$ be a homogeneous space. Then $X$ is either a meager space or a Baire space.
\end{proposition}

\begin{proposition}\label{locallycompact}
Let $X$ be a zero-dimensional locally compact space. Then $X$ is homogeneous iff $X$ is discrete, $X\approx 2^\omega$, or $X\approx\omega\times 2^\omega$.
\end{proposition}
\begin{proof}
The right-to-left implication is trivial. For the left-to-right implication, use the well-known characterization of $2^\omega$ as the unique zero-dimensional crowded compact space (see \cite[Theorem 7.4]{kechris}).
\end{proof}

\begin{proposition}\label{somewherepolish}
Let $X$ be a zero-dimensional homogeneous space. If there exists a non-empty Polish $U\in\mathbf{\Sigma}^0_1(X)$ then $X$ is Polish.
\end{proposition}
\begin{proof}
Let $U\in\mathbf{\Sigma}^0_1(X)$ be non-empty and Polish. Since $X$ is zero-dimensional, we can assume without loss of generality that $U\in\mathbf{\Delta}^0_1(X)$. Let $\UU=\{h[U]:h\text{ is a homeomorphism of }X\}$. Notice that $\UU$ is a cover of $X$ because $X$ is homogeneous and $U$ is non-empty. Let $\{U_n:n\in\omega\}$ be a countable subcover of $\UU$. Define $V_n=U_n\setminus\bigcup_{k<n}U_k$ for $n\in\omega$, and observe that each $V_n$ is Polish. Since $V_n\cap V_m=\varnothing$ whenever $m\neq n$, it follows from \cite[Proposition 3.3.iii]{kechris} that $X=\bigcup_{n\in\omega}V_n$ is Polish.
\end{proof}

\begin{proposition}\label{bairecomeager}
Assume $\BP$. Let $Z$ be a Polish space, and let $X$ be a dense Baire subspace of $Z$. Then $X$ is comeager in $Z$.
\end{proposition}
\begin{proof}
Since $X$ has the Baire property, we can write $X=G\cup M$ by \cite[Proposition 8.23.ii]{kechris}, where $G\in\mathbf{\Pi}^0_2(Z)$ and $M$ is meager in $Z$. It will be enough to show that $G$ is dense in $Z$. Assume, in order to get a contradiction, that there exists a non-empty open subset $U$ of $Z$ such that $U\cap G=\varnothing$. Observe that $U\cap X$ is a non-empty open subset of $X$ because $X$ is dense in $Z$. Furthermore, using the density of $X$ again, it is easy to see that $M=M\cap X$ is meager in $X$. Since $U\cap X\subseteq M$, this contradicts the fact that $X$ is a Baire space.
\end{proof}

\begin{theorem}[Terada]\label{pibase}
Let $X$ be a non-compact space. Assume that $X$ has a base $\BB\subseteq\mathbf{\Delta}^0_1(X)$ such that $U\approx X$ for every $U\in\BB$. Then $X$ is strongly homogeneous.
\end{theorem}

\section{The basics of Wadge theory}

The following simple lemma will allow us to generalize many Wadge-theoretic results from $\omega^\omega$ to an arbitrary zero-dimensional Polish space. This approach has already appeared in \cite[Section 5]{andretta}, where it is credited to Marcone. Recall that, given a space $Z$ and $W\subseteq Z$, a \emph{retraction} is a continuous function $\rho:Z\longrightarrow W$ such that $\rho\re W=\id_W$. By \cite[Theorem 7.8]{kechris}, every zero-dimensional Polish space is homeomorphic to a closed subspace $Z$ of $\omega^\omega$, and by \cite[Proposition 2.8]{kechris} there exists a retraction $\rho:\omega^\omega\longrightarrow Z$.

\begin{lemma}\label{bairetoall}
Let $Z\subseteq\omega^\omega$, and let $\rho:\omega^\omega\longrightarrow Z$ be a retraction. Fix $A,B\subseteq Z$. Then $A\leq B$ in $Z$ iff $\rho^{-1}[A]\leq\rho^{-1}[B]$ in $\omega^\omega$. 
\end{lemma}
\begin{proof}
If $f:Z\longrightarrow Z$ witnesses that $A\leq B$ in $Z$, then $f\circ\rho:\omega^\omega\longrightarrow\omega^\omega$ will witness that $\rho^{-1}[A]\leq\rho^{-1}[B]$ in $\omega^\omega$. On the other hand, if $f:\omega^\omega\longrightarrow\omega^\omega$ witnesses that $\rho^{-1}[A]\leq\rho^{-1}[B]$ in $\omega^\omega$, then $\rho\circ(f\re Z):Z\longrightarrow Z$ will witness that $A\leq B$ in $Z$.
\end{proof}

The following result (commonly known as ``Wadge's Lemma'') shows that antichains with respect to $\leq$ have size at most $2$.
\begin{lemma}[Wadge]\label{wadgelemma}
Assume $\AD$. Let $Z$ be a zero-dimensional Polish space, and let $A,B\subseteq Z$. Then either $A\leq B$ or $Z\setminus B\leq A$.
\end{lemma}
\begin{proof}
For the case $Z=\omega^\omega$, see \cite[proof of Theorem 21.14]{kechris}. To obtain the full result from this particular case, use Lemma \ref{bairetoall} and the remarks preceding it.
\end{proof}

\begin{theorem}[Martin, Monk]\label{wellfounded}
Assume $\AD$. Let $Z$ be a zero-dimensional Polish space. Then the relation $\leq$ on $\PP(Z)$ is well-founded.
\end{theorem}
\begin{proof}
For the case $Z=\omega^\omega$, see \cite[proof of Theorem 21.15]{kechris}. To obtain the full result from this particular case, use Lemma \ref{bairetoall} and the remarks preceding it.
\end{proof}

Given a zero-dimensional Polish space $Z$, define
$$
\W(Z)=\{\{\bG,\bGc\}:\bG\text{ is a Wadge class in }Z\}.
$$
Given $p,q\in\W(Z)$, define $p\prec q$ if $\bG\subseteq\bL$ for every $\bG\in p$ and $\bL\in q$. Using the two previous results, one sees that the ordering $\prec$ on $\W(Z)$ is a well-order. Therefore, there exists an order-isomorphism $\phi:\W(Z)\longrightarrow\Theta$ for some ordinal $\Theta$.\footnote{\,For a characterization of $\Theta$, see \cite[Definition 0.1 and Lemma 0.2]{solovay}.} The reason for the ``1+'' in the definition below is simply a matter of technical convenience (see \cite[page 45]{andrettahjorthneeman}).
\begin{definition}\label{wadgerank}
Let $Z$ be a zero-dimensional Polish space, and let $\bG$ be a Wadge class in $Z$. Define
$$
||\bG||=1+\phi(\{\bG,\bGc\}).
$$
We will say that $||\bG||$ is the \emph{Wadge-rank} of $\bG$.
\end{definition}

It is easy to check that $\{\{\varnothing\},\{Z\}\}$ is the minimal element of $\W(Z)$. Furthermore, elements of the form $\{\bG,\bGc\}$ for $\bG\in\NSD(Z)$ are always followed by $\{\bD\}$ for some $\bD\in\SD(Z)$, while elements of the form $\{\bD\}$ for $\bD\in\SD(Z)$ are always followed by $\{\bG,\bGc\}$ for some $\bG\in\NSD(Z)$. This was proved by Van Wesep for $Z=\omega^\omega$ (see \cite[Corollary to Theorem 2.1]{vanwesept}), and it can be generalized to arbitrary uncountable zero-dimensional Polish spaces using Corollary \ref{selfdualcorollary} and the machinery of relativization that we will develop in Sections 6 to 8. Since these facts will not be needed in the remainder of the paper, we omit the proof.

In fact, as Proposition \ref{orderisomorphism} (together with Theorem \ref{hausdorffmain}) will show, the ordering of the non-selfdual classes is independent of the space $Z$. However, the situation is more delicate for selfdual classes. For example, it follows easily from Corollary \ref{selfdualcorollary} that if $\bG$ is a Wadge class in $2^\omega$ such that $||\bG||$ is a limit ordinal of countable cofinality, then $\bG$ is non-selfdual. On the other hand, if $\bG$ is a Wadge class in $\omega^\omega$ such that $||\bG||$ is a limit ordinal of countable cofinality, then $\bG$ is selfdual (see \cite[Corollary to Theorem 2.1]{vanwesept} again).

The collection of all Wadge classes on a given space $Z$, ordered by $\subseteq$, is known as the \emph{Wadge hierarchy}. The following diagram shows how this hierarchy looks like when $Z$ is a zero-dimensional Polish space.\footnote{\,Notice that this hierachy can collapse rather soon when $Z$ is countable. For example, when $Z$ is the discrete space $\omega$, the Wadge hierarchy consists only of the three bottom classes.}
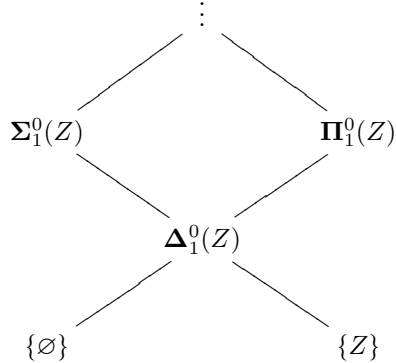
\begin{center}
$
\xymatrix{
&\,\,\,\,\,\vdots\,\,\,\,\,\ar@{-}[rd]\ar@{-}[ld]&\\
\mathbf{\Sigma}^0_1(Z)\ar@{-}[rd]& &\mathbf{\Pi}^0_1(Z)\ar@{-}[ld]\\
&\bD^0_1(Z)\ar@{-}[rd]\ar@{-}[ld]&\\
\{\varnothing\}& &\{Z\}
}
$
\end{center}
\bigskip

We conclude this section with an elementary result, which shows that clopen sets are ``neutral sets'' for Wadge-reduction (the simple proof is left to the reader). By this we mean that, apart from trivial exceptions, intersections or unions with these sets do not change the Wadge class. In Section 12, we will prove more sophisticated closure properties.

\begin{proposition}\label{closureclopen}
Let $Z$ be a space, let $\bG$ be a Wadge class in $Z$, and let $A\in\bG$.
\begin{itemize}
\item Assume that $\bG\neq\{Z\}$. Then $A\cap V\in\bG$ for every $V\in\mathbf{\Delta}^0_1(Z)$.
\item Assume that $\bG\neq\{\varnothing\}$. Then $A\cup V\in\bG$ for every $V\in\mathbf{\Delta}^0_1(Z)$.
\end{itemize}
\end{proposition}

\section{The analysis of selfdual sets}

In this section we will simply collect well-known results which show that every selfdual set can be built using non-selfdual sets of lower complexity (apply Corollary \ref{selfdualcorollary} with $V=Z$). We will refer to the proof of \cite[Theorem 5.3]{mottoros}, which in turn generalizes \cite[Theorem 16]{andrettamartin} (see also \cite[Lemma 7.3.4]{louveaub}).

\begin{theorem}\label{selfdualtheorem}
Assume $\BP$. Let $Z$ be a zero-dimensional Polish space, let $V\in\bD^0_1(Z)$, and let $A$ be a selfdual subset of $Z$. Assume that $A\notin\mathbf{\Delta}^0_1(Z)$. Then there exist pairwise disjoint $V_n\in\mathbf{\Delta}^0_1(V)$ for $n\in\omega$ such that $\bigcup_{n\in\omega}V_n=V$ and $A\cap V_n< A$ in $Z$ for each $n$.
\end{theorem}
\begin{proof}
This is proved like \cite[Theorem 5.3]{mottoros}, with $Z$ instead of $\omega^\omega$ (which is denoted by $\RRR$ there) and $D_0=V$, where $\FF$ is the collection of all continuous $f:Z\longrightarrow Z$ and $\Delta_\FF=\mathbf{\Delta}^0_1(Z)$.
\end{proof}

\begin{corollary}\label{selfdualcorollary}
Assume $\AD$. Let $Z$ be a zero-dimensional Polish space, let $V\in\bD^0_1(Z)$, and let $A$ be a selfdual subset of $Z$. Then there exist pairwise disjoint $V_n\in\mathbf{\Delta}^0_1(V)$ and non-selfdual $A_n<A$ in $Z$ for $n\in\omega$ such that $\bigcup_{n\in\omega}V_n=V$ and $\bigcup_{n\in\omega}(A_n\cap V_n)=A\cap V$.
\end{corollary}
\begin{proof}
As one can easily check, it will be enough to show that there exist pairwise disjoint $V_n\in\mathbf{\Delta}^0_1(V)$ for $n\in\omega$ such that $\bigcup_{n\in\omega}V_n=V$ and for every $n\in\omega$ either $A\cap V_n\in\mathbf{\Delta}^0_1(Z)$ or $A\cap V_n$ is non-selfdual in $Z$. If this were not the case, then, using Theorem \ref{selfdualtheorem}, one could recursively construct a strictly $\leq$-decreasing sequence of subsets of $Z$, which would contradict Theorem \ref{wellfounded}.
\end{proof}

\section{Basic facts on Hausdorff operations}

For a history of the following important notion, see \cite[page 583]{hausdorff}. For a modern survey, we recommend \cite{zafrany}. Most of the proofs in this section are straightforward, hence we leave them to the reader.

\begin{definition}
Given a set $Z$ and $D\subseteq\PP(\omega)$, define
$$
\HH_D(A_0,A_1,\ldots)=\{x\in Z:\{n\in\omega:x\in A_n\}\in D\}
$$
whenever $A_0,A_1,\ldots\subseteq Z$. Functions of this form are called \emph{Hausdorff operations} (or \emph{$\omega$-ary Boolean operations}).
\end{definition}

Of course, the function $\HH_D$ depends on the set $Z$, but what $Z$ is will usually be clear from the context. In case there might be uncertainty about the ambient space, we will use the notation $\HH_D^Z$. Notice that, once $D$ is specified, the corresponding Hausdorff operation simultaneously defines functions $\PP(Z)^\omega\longrightarrow\PP(Z)$ for every $Z$.

The following proposition lists the most basic properties of Hausdorff operations. Given $n\in\omega$, define $S_n=\{A\subseteq\omega:n\in A\}$.

\begin{proposition}\label{hausdorffsettheoretic}
Let $I$ be a non-empty set, and let $D_i\subseteq\PP(\omega)$ for every $i\in I$. Fix an ambient set $Z$ and $A_0,A_1,\ldots\subseteq Z$.
\begin{itemize}
\item $\HH_{S_n}(A_0,A_1,\ldots)=A_n$ for all $n\in\omega$.
\item $\bigcap_{i\in I}\HH_{D_i}(A_0,A_1,\ldots)=\HH_D(A_0,A_1,\ldots)$, where $D=\bigcap_{i\in I}D_i$.
\item $\bigcup_{i\in I}\HH_{D_i}(A_0,A_1,\ldots)=\HH_D(A_0,A_1,\ldots)$, where $D=\bigcup_{i\in I}D_i$.
\item $Z\setminus\HH_D(A_0,A_1,\ldots)=\HH_{\PP(\omega)\setminus D}(A_0,A_1,\ldots)$ for all $D\subseteq\PP(\omega)$.
\end{itemize}
\end{proposition}
The point of the above proposition is that any operation obtained by combining unions, intersections and complements can be expressed as a Hausdorff operation. For example, if $D=\bigcup_{n\in\omega}(S_{2n+1}\setminus S_{2n})$, then $\HH_D(A_0,A_1,\ldots)=\bigcup_{n\in\omega}(A_{2n+1}\setminus A_{2n})$.

The following proposition shows that the composition of Hausdorff operations is again a Hausdorff operation. We will assume that a bijection $\pi:\omega\times\omega\longrightarrow\omega$ has been fixed, and use the notation $\langle m,n\rangle=\pi(m,n)$.
\begin{proposition}\label{hausdorffcomposition}
Let $Z$ be a set, let $D\subseteq\PP(\omega)$ and $E_m\subseteq\PP(\omega)$ for $m\in\omega$. Then there exists $F\subseteq\PP(\omega)$ such that
$$
\HH_D(B_0,B_1,\ldots)=\HH_F(A_0,A_1,\ldots)
$$
for all $A_0,A_1,\ldots\subseteq Z$, where $B_m=\HH_{E_m}(A_{\langle m,0\rangle},A_{\langle m,1\rangle},\ldots)$.
\end{proposition}
\begin{proof}
Define $z\in F$ if $\{m\in\omega:\{n\in\omega:\langle m,n\rangle\in z\}\in E_m\}\in D$. The rest of the proof is a straightforward verification.
\end{proof}

We conclude this section with a result that will easily imply the fundamental Lemma \ref{relativization}.
\begin{proposition}\label{relativizationsettheoretic}
Let $Z$ and $W$ be sets, let $D\subseteq\PP(\omega)$, let $A_0,A_1,\ldots\subseteq Z$ and $B_0,B_1,\ldots\subseteq W$.
\begin{enumerate}
\item\label{subspacesettheoretic} $W\cap\HH_D^Z(A_0,A_1,\ldots)=\HH_D^W(A_0\cap W,A_1\cap W,\ldots)$ whenever $W\subseteq Z$.
\item\label{preimagesettheoretic} $f^{-1}[\HH_D(B_0,B_1,\ldots)]=\HH_D(f^{-1}[B_0],f^{-1}[B_1],\ldots)$ for all $f:Z\longrightarrow W$.
\item\label{homeomorphismsettheoretic} $f[\HH_D(A_0,A_1,\ldots)]=\HH_D(f[A_0],f[A_1],\ldots)$ for all bijections $f:Z\longrightarrow W$.
\end{enumerate}
\end{proposition}

\section{Wadge classes and Hausdorff operations}

When one tries to give a systematic exposition of Wadge theory, it soon becomes apparent that it would be very useful to be able to talk about ``abstract'' Wadge classes, as opposed to Wadge classes in a particular space. More precisely, given a Wadge class $\bG$ in some space $Z$, one would like to find a way to define what a ``$\bG$ subset of $W$'' is, for every other space $W$, while of course preserving suitable coherence properties. It turns out that Hausdorff operations allow us to do exactly that in a rather elegant way, provided that $\bG$ is a non-selfdual Wadge class, $Z$ and $W$ are uncountable zero-dimensional Polish spaces, and $\AD$ holds (see also the discussion in Section 3). For an early instance of this idea, see \cite[Theorem 4.2]{louveausaintraymondc}.\footnote{\,This result is limited to the Borel context. On the other hand, the ambient space is allowed to be analytic, as opposed to Polish.} The following is the crucial definition. In fact, the aim of this section and the next two is to show that the classes $\bG_D(Z)$ have nice properties (see Lemma \ref{relativization} and Proposition \ref{orderisomorphism}), and that, under $\AD$, they are exactly the non-selfdual Wadge classes on $Z$ (see Theorem \ref{hausdorffmain}).

\begin{definition}
Given a space $Z$ and $D\subseteq\PP(\omega)$, define
$$
\bG_D(Z)=\{\HH_D(A_0,A_1,\ldots):A_n\in\mathbf{\Sigma}^0_1(Z)\text{ for every }n\in\omega\}.
$$
\end{definition}

As examples (that will be useful later), consider the following two simple propositions.

\begin{proposition}\label{hausdorffdifferences}
Let $1\leq\eta<\omega_1$. Then there exists $D\subseteq\PP(\omega)$ such that $\bG_D(Z)=\mathsf{D}_\eta(\mathbf{\Sigma}^0_1(Z))$ for every space $Z$.
\end{proposition}
\begin{proof}
This follows from Propositions \ref{hausdorffsettheoretic} and \ref{hausdorffcomposition} (in case $\eta>\omega$, use a bijection $\pi:\eta\longrightarrow\omega$).
\end{proof}

\begin{proposition}\label{hausdorffborel}
Let $1\leq\xi<\omega_1$. Then there exists $D\subseteq\PP(\omega)$ such that $\bG_D(Z)=\mathbf{\Sigma}^0_\xi(Z)$ for every space $Z$.
\end{proposition}
\begin{proof}
This can be proved by induction on $\xi$, using Propositions \ref{hausdorffsettheoretic} and \ref{hausdorffcomposition}.
\end{proof}

Next, we obtain a very useful lemma, which shows that this notion behaves well with respect to subspaces and continuous functions. This lemma is essentially what we refer to when we speak about the ``machinery of relativization''. It extends (and is inspired by) \cite[Lemma 2.3]{vanengelena}.

\begin{lemma}\label{relativization}
Let $Z$ and $W$ be spaces, and let $D\subseteq\PP(\omega)$.
\begin{enumerate}
\item\label{subspace} Assume that $W\subseteq Z$. Then $B\in\bG_D(W)$ iff there exists $A\in\bG_D(Z)$ such that $B=A\cap W$.
\item\label{preimage} If $f:Z\longrightarrow W$ is continuous and $B\in\bG_D(W)$ then $f^{-1}[B]\in\bG_D(Z)$.
\item\label{homeomorphism} If $h:Z\longrightarrow W$ is a homeomorphism then $A\in\bG_D(Z)$ iff $h[A]\in\bG_D(W)$.
\end{enumerate}
\end{lemma}
\begin{proof}
This is a straightforward consequence of Proposition \ref{relativizationsettheoretic}.
\end{proof}

The following simple result, together with Theorem \ref{hausdorffmain}, shows that the ordering of the non-selfdual Wadge classes is independent of the ambient space $Z$ (provided that $\AD$ holds).

\begin{proposition}\label{orderisomorphism}
Let $Z$ and $W$ be zero-dimensional spaces that contain a copy of $2^\omega$, and let $D,E\subseteq\PP(\omega)$. Then $\bG_D(Z)\subseteq\bG_E(Z)$ iff $\bG_D(W)\subseteq\bG_E(W)$.
\end{proposition}
\begin{proof}
Assume that $\bG_D(Z)\subseteq\bG_E(Z)$. Since $Z$ contains a copy of $2^\omega$ and $W$ is zero-dimensional, we see that $Z$ contains a copy of $W$. Using Lemma \ref{relativization}.\ref{homeomorphism}, we can assume without loss of generality that $W\subseteq Z$. Then
$$
\bG_D(W)=\{A\cap W:A\in\bG_D(Z)\}\subseteq\{A\cap W:A\in\bG_E(Z)\}=\bG_E(W),
$$
where the first and last equalities hold by Lemma \ref{relativization}.\ref{subspace}. The proof of the other implication is similar.
\end{proof}

\section{Universal sets}

The aim of this section is to prove the easier half of Theorem \ref{hausdorffmain} (namely, Theorem \ref{addison}). The ideas presented here are well-known, but since we could not find a satisfactory reference, we will give all the details. Our approach is inspired by \cite[Section 22.A]{kechris}. 

\begin{definition}
Let $Z$ and $W$ be spaces, and let $D\subseteq\PP(\omega)$. Given $U\subseteq W\times Z$ and $x\in W$, let $U_x=\{y\in Z:(x,y)\in U\}$ denote the vertical section of $U$ above $x$. We will say that $U\subseteq W\times Z$ is a \emph{$W$-universal set} for $\bG_D(Z)$ if the following two conditions hold:
\begin{itemize}
\item $U\in\bG_D(W\times Z)$,
\item $\{U_x:x\in W\}=\bG_D(Z)$.
\end{itemize}	
\end{definition}

Notice that, by Proposition \ref{hausdorffborel}, the above yields the definition of a $W$-universal set for $\mathbf{\Sigma}^0_\xi(Z)$ whenever $1\leq\xi <\omega_1$. Furthermore, this definition agrees with \cite[Definition 22.2]{kechris}.

\begin{proposition}\label{existsCuniversal}
Let $Z$ be a space, and let $D\subseteq\PP(\omega)$. Then there exists a $2^\omega$-universal set for $\bG_D(Z)$.
\end{proposition}
\begin{proof}
By \cite[Theorem 22.3]{kechris}, we can fix a $2^\omega$-universal set $U$ for $\mathbf{\Sigma}^0_1(Z)$. Let $h:2^\omega\longrightarrow (2^\omega)^\omega$ be a homeomorphism, and let $\pi_n:(2^\omega)^\omega\longrightarrow 2^\omega$ be the projection on the $n$-th coordinate for $n\in\omega$. Notice that, given any $n\in\omega$, the function $f_n:2^\omega\times Z\longrightarrow 2^\omega\times Z$ defined by $f_n(x,y)=(\pi_n(h(x)),y)$ is continuous. Let $V_n=f_n^{-1}[U]$ for each $n$, and observe that each $V_n\in\mathbf{\Sigma}^0_1(2^\omega\times Z)$. Set $V=\HH_D(V_0,V_1,\ldots)$.

We claim that $V$ is a $2^\omega$-universal set for $\bG_D(Z)$. It is clear that $V\in\bG_D(2^\omega\times Z)$. Furthermore, using Lemma \ref{relativization}, one can easily check that $V_x\in\bG_D(Z)$ for every $x\in 2^\omega$. To complete the proof, fix $A\in\bG_D(Z)$. Let $A_0,A_1,\ldots\in\mathbf{\Sigma}^0_1(Z)$ be such that $A=\HH_D(A_0,A_1,\ldots)$. Since $U$ is $2^\omega$-universal, we can fix $z_n\in 2^\omega$ such that $U_{z_n}=A_n$ for every $n\in\omega$. Set $z=h^{-1}(z_0,z_1,\ldots)$. It is straightforward to verify that $V_z=A$.
\end{proof}

\begin{corollary}\label{existsZuniversal}
Let $Z$ be a space that contains a copy of $2^\omega$, and let $D\subseteq\PP(\omega)$. Then there exists a $Z$-universal set for $\bG_D(Z)$.
\end{corollary}
\begin{proof}
By Proposition \ref{existsCuniversal}, we can fix a $2^\omega$-universal set $U$ for $\bG_D(Z)$. Let $W\subseteq Z$ be such that $W\approx 2^\omega$, and fix a homeomorphism $h:2^\omega\longrightarrow W$. Notice that $(h\times\id_Z)[U]\in\bG_D(W\times Z)$ by Lemma \ref{relativization}.\ref{homeomorphism}. Therefore, by Lemma \ref{relativization}.\ref{subspace}, there exists $V\in\bG_D(Z\times Z)$ such that $V\cap (W\times Z)=(h\times\id_Z)[U]$. Using Lemma \ref{relativization} again, one can easily check that $V$ is a $Z$-universal set for $\bG_D(Z)$.
\end{proof}

\begin{lemma}\label{Zuniversalnonselfdual}
Let $Z$ be a space, and let $D\subseteq\PP(\omega)$. Assume that there exists a $Z$-universal set for $\bG_D(Z)$. Then $\bG_D(Z)$ is non-selfdual.
\end{lemma}
\begin{proof}
Fix a $Z$-universal set $U\subseteq Z\times Z$ for $\bG_D(Z)$. Assume, in order to get a contradiction, that $\bG_D(Z)$ is selfdual. Let $f:Z\longrightarrow Z\times Z$ be the function defined by $f(x)=(x,x)$, and observe that $f$ is continuous. Since $f^{-1}[U]\in\bG_D(Z)=\bGc_D(Z)$, we see that $Z\setminus f^{-1}[U]\in\bG_D(Z)$. Therefore, since $U$ is $Z$-universal, we can fix $z\in Z$ such that $U_z=Z\setminus f^{-1}[U]$. If $z\in U_z$ then $f(z)=(z,z)\in U$ by the definition of $U_z$, contradicting the fact that $U_z=Z\setminus f^{-1}[U]$. On the other hand, If $z\notin U_z$ then $f(z)=(z,z)\notin U$ by the definition of $U_z$, contradicting the fact that $Z\setminus U_z=f^{-1}[U]$.
\end{proof}

The case $Z=\omega^\omega$ of the following result is \cite[Proposition 5.0.3]{vanwesept}, and it is credited to Addison by Van Wesep.
\begin{theorem}\label{addison}
Let $Z$ be a zero-dimensional space that contains a copy of $2^\omega$, and let $D\subseteq\PP(\omega)$. Then $\bG_D(Z)\in\NSD(Z)$.
\end{theorem}
\begin{proof}
The fact that $\bG_D(Z)$ is non-selfdual follows from Corollary \ref{existsZuniversal} and Lemma \ref{Zuniversalnonselfdual}. Therefore, it will be enough to show that $\bG_D(Z)$ is a Wadge class. By Proposition \ref{existsCuniversal}, we can fix a $2^\omega$-universal set $U\subseteq 2^\omega\times Z$ for $\bG_D(Z)$. Let $W\subseteq Z$ be such that $W\approx 2^\omega\times Z$, and fix a homeomorphism $h:2^\omega\times Z\longrightarrow W$. By Lemma \ref{relativization}, we can fix $A\in\bG_D(Z)$ such that $A\cap W=h[U]$. We claim that $\bG_D(Z)=[A]$. The inclusion $\supseteq$ follows from Lemma \ref{relativization}.\ref{preimage}. In order to prove the other inclusion, pick $B\in\bG_D(Z)$. Since $U$ is $2^\omega$-universal, we can fix $z\in 2^\omega$ such that $B=U_z$. Consider the function $f:Z\longrightarrow 2^\omega\times Z$ defined by $f(x)=(z,x)$, and observe that $f$ is continuous. It is straightforward to check that $h\circ f:Z\longrightarrow Z$ witnesses that $B\leq A$ in $Z$.
\end{proof}	

\section{Van Wesep's theorem}

The following is one of the main results of Van Wesep's doctoral thesis (see \cite[Theorem 5.3.1]{vanwesept}, whose proof also employs results of Steel from \cite{steelt} and results of Radin), and it will allow us to obtain the harder half of Theorem \ref{hausdorffmain}. Notice how Corollary \ref{hausdorffevery} guarantees that every non-selfdual Wadge class is amenable to the machinery of relativization.

\begin{theorem}[Van Wesep]\label{hausdorffbaire}
Assume $\AD$. For every $\bG\in\NSD(\omega^\omega)$ there exists $D\subseteq\PP(\omega)$ such that $\bG=\bG_D(\omega^\omega)$.
\end{theorem}
\begin{corollary}\label{hausdorffevery}
Assume $\AD$. Let $Z$ be a zero-dimensional Polish space, and let $\bG\in\NSD(Z)$. Then there exists $D\subseteq\PP(\omega)$ such that $\bG=\bG_D(Z)$.	
\end{corollary}
\begin{proof}
By \cite[Theorem 7.8]{kechris}, there exists a closed $W\subseteq\omega^\omega$ such that $Z\approx W$. Therefore, using Lemma \ref{relativization}.\ref{homeomorphism}, we can assume without loss of generality that $Z$ is a closed subspace of $\omega^\omega$. Hence, by \cite[Proposition 2.8]{kechris}, we can fix a retraction $\rho:\omega^\omega\longrightarrow Z$. Let $A\subseteq Z$ be such that $\bG=[A]$. Set $B=\rho^{-1}[A]$, and let $\bL=[B]$ be the Wadge class generated by $B$ in $\omega^\omega$.

Using Lemma \ref{bairetoall}, it is easy to see that $\bL\in\NSD(\omega^\omega)$. Therefore, by Theorem \ref{hausdorffbaire}, we can fix $D\subseteq\PP(\omega)$ such that $\bL=\bG_D(\omega^\omega)$. We claim that $\bG=\bG_D(Z)$. Notice that $A=B\cap Z\in\bG_D(Z)$ by Lemma \ref{relativization}.\ref{subspace}, hence $\bG\subseteq\bG_D(Z)$ by Lemma \ref{relativization}.\ref{preimage}. Finally, to see that $\bG_D(Z)\subseteq\bG$, pick $C\in\bG_D(Z)$. Observe that $\rho^{-1}[C]\in\bG_D(\omega^\omega)=\bL$ by Lemma \ref{relativization}.\ref{preimage}. This means that $\rho^{-1}[C]\leq B=\rho^{-1}[A]$ in $\omega^\omega$, hence $C\leq A$ in $Z$ by Lemma \ref{bairetoall}. So $C\in [A]=\bG$, which concludes the proof.
\end{proof}

Finally, we can ``put everything together'' and state the full result promised in the introduction to Section 6.

\begin{theorem}\label{hausdorffmain}
Assume $\AD$. Let $Z$ be an uncountable zero-dimensional Polish space. Then
$$
\NSD(Z)=\{\bG_D(Z):D\subseteq\PP(\omega)\}
$$	
\end{theorem}
\begin{proof}
This follows immediately from Theorem \ref{addison} and Corollary \ref{hausdorffevery}.
\end{proof}

\section{Basic facts on expansions}

The following notion is essentially due to Wadge (see \cite[Chapter IV]{wadget}), and it is inspired by work of Kuratowski. Recall that, given $1\leq\xi<\omega_1$ and spaces $Z$ and $W$, a function $f:Z\longrightarrow W$ is \emph{$\mathbf{\Sigma}^0_\xi$-measurable} if $f^{-1}[U]\in\mathbf{\Sigma}^0_\xi(Z)$ for every $U\in\mathbf{\Sigma}^0_1(W)$.

\begin{definition}
Let $Z$ be a space, and let $\xi<\omega_1$. Given $\bG\subseteq\PP(Z)$, define
$$
\bG^{(\xi)}=\{f^{-1}[A]:A\in\bG\text{ and }f:Z\longrightarrow Z\text{ is $\mathbf{\Sigma}^0_{1+\xi}$-measurable}\}.
$$
We will refer to $\bG^{(\xi)}$ as an \emph{expansion} of $\bG$.
\end{definition}

The following is the corresponding definition in the context of Hausdorff operations. Corollary \ref{movexi} below shows that this is in fact the ``right'' definition.
\begin{definition}\label{hausdorffexpansion}
Let $Z$ be a space, let $D\subseteq\PP(\omega)$, and let $\xi<\omega_1$. Define
$$
\bG_D^{(\xi)}(Z)=\{\HH_D(A_0,A_1,\ldots):A_n\in\mathbf{\Sigma}^0_{1+\xi}(Z)\text{ for every }n\in\omega\}.
$$
\end{definition}

As an example (that will be useful later), consider the following simple observation.

\begin{proposition}\label{expansiondifferences}
Let $1\leq\eta<\omega_1$. Then there exists $D\subseteq\PP(\omega)$ such that $\bG_D^{(\xi)}(Z)=\mathsf{D}_\eta(\mathbf{\Sigma}^0_{1+\xi}(Z))$ for every space $Z$ and every $\xi<\omega_1$.
\end{proposition}
\begin{proof}
This is proved like Proposition \ref{hausdorffdifferences} (in fact, the same $D$ will work).
\end{proof}

The following proposition shows that Definition \ref{hausdorffexpansion} actually fits in the context provided by Section 6.

\begin{proposition}\label{movexidown}
Let $D\subseteq\PP(\omega)$, and let $\xi<\omega_1$. Then there exists $E\subseteq\PP(\omega)$ such that $\bG_D^{(\xi)}(Z)=\bG_E(Z)$ for every space $Z$.
\end{proposition}
\begin{proof}
This is proved by combining Propositions \ref{hausdorffborel} and \ref{hausdorffcomposition}.
\end{proof}
\begin{corollary}\label{expansionhausdorffnonselfdual}
Let $Z$ be an uncountable zero-dimensional Polish space, let $D\subseteq\PP(\omega)$, and let $\xi<\omega_1$. Then $\bG_D^{(\xi)}(Z)\in\NSD(Z)$.
\end{corollary}
\begin{proof}
This is proved by combining Proposition \ref{movexidown} and Theorem \ref{addison}.
\end{proof}

The following useful result is the analogue of Lemma \ref{relativization} in the present context.

\begin{lemma}\label{expansionrelativization}
Let $Z$ and $W$ be spaces, let $D\subseteq\PP(\omega)$, and let $\xi<\omega_1$.
\begin{enumerate}
\item\label{expansionsubspace} Assume that $W\subseteq Z$. Then $B\in\bG_D^{(\xi)}(W)$ iff there exists $A\in\bG_D^{(\xi)}(Z)$ such that $B=A\cap W$.
\item\label{expansionpreimage} If $f:Z\longrightarrow W$ is continuous and $B\in\bG_D^{(\xi)}(W)$ then $f^{-1}[B]\in\bG_D^{(\xi)}(Z)$.
\item\label{expansionmeasurablepreimage} If $f:Z\longrightarrow W$ is $\mathbf{\Sigma}^0_{1+\xi}$-measurable and $B\in\bG_D(W)$ then $f^{-1}[B]\in\bG_D^{(\xi)}(Z)$.
\item\label{expansionhomeomorphism} If $h:Z\longrightarrow W$ is a homeomorphism then $A\in\bG_D^{(\xi)}(Z)$ iff $h[A]\in\bG_D^{(\xi)}(W)$.
\end{enumerate}
\end{lemma}
\begin{proof}
This is a straightforward consequence of Proposition \ref{relativizationsettheoretic}.
\end{proof}

\section{Kuratowski's transfer theorem}

The aim of this section is to collect the tools needed to successfully employ the notion of expansion. For example, Corollary \ref{expansionbijection} will be a crucial ingredient in the proof of Theorem \ref{expansiontheorem}. A stronger form of Theorem \ref{sigmaintoopen} appears as \cite[Theorem 7.1.6]{louveaub}, where it is called ``Kuratowski's transfer theorem''.

\begin{theorem}[Kuratowski]\label{sigmaintoopen}
Let $(Z,\tau)$ be a Polish space, let $1<\xi<\omega_1$, and let $\Aa\subseteq\mathbf{\Sigma}^0_\xi(Z,\tau)$ be countable. Then there exists a zero-dimensional Polish topology $\sigma$ on the set $Z$ such that $\tau\subseteq\sigma\subseteq\mathbf{\Sigma}^0_\xi(Z,\tau)$ and $\Aa\subseteq\sigma$.
\end{theorem}
\begin{proof}
This follows from \cite[Exercise 22.20]{kechris}, using the fact that every element of $\mathbf{\Sigma}^0_\xi(Z,\tau)$ can be written as a countable union of elements of $\bD^0_\xi(Z,\tau)$.
\end{proof}

\begin{corollary}\label{sigmaintoopenbijection}
Let $Z$ be a zero-dimensional Polish space, let $1\leq\xi<\omega_1$, and let $\Aa\subseteq\mathbf{\Sigma}^0_\xi(Z)$ be countable. Then there exists a zero-dimensional Polish space $W$ and a $\mathbf{\Sigma}^0_\xi$-measurable bijection $f:Z\longrightarrow W$ such that $f[A]\in\mathbf{\Sigma}^0_1(W)$ for every $A\in\Aa$.
\end{corollary}
\begin{proof}
The case $\xi=1$ is trivial, so assume that $\xi>1$. The space $W$ is simply the set $Z$ with the finer topology given by Theorem \ref{sigmaintoopen}, while $f=\id_Z$.
\end{proof}

\begin{corollary}\label{expansionbijection}
Let $Z$ be a zero-dimensional Polish space, let $D\subseteq\PP(\omega)$, and let $\xi<\omega_1$. Assume that $\Aa\subseteq\bG_D^{(\xi)}(Z)$ and $\BB\subseteq\mathbf{\Sigma}^0_{1+\xi}(Z)$ are countable. Then there exists a zero-dimensional Polish space $W$ and a $\mathbf{\Sigma}^0_{1+\xi}$-measurable bijection $f:Z\longrightarrow W$ such that $f[A]\in\bG_D(W)$ for every $A\in\Aa$ and $f[B]\in\mathbf{\Sigma}^0_1(W)$ for every $B\in\BB$.
\end{corollary}
\begin{proof}
Let $\Aa=\{A_m:m\in\omega\}$. Given $m\in\omega$, fix $A_{m,n}\in\mathbf{\Sigma}^0_{1+\xi}(Z)$ for $n\in\omega$ such that $A_m=\HH_D(A_{m,0},A_{m,1},\ldots)$. Define $\CC=\{A_{m,n}:m,n\in\omega\}\cup\BB$. By Corollary \ref{sigmaintoopenbijection}, we can fix a Polish space $W$ and a $\mathbf{\Sigma}^0_{1+\xi}$-measurable bijection $f:Z\longrightarrow W$ such that $f[C]\in\mathbf{\Sigma}^0_1(W)$ for every $C\in\CC$. It remains to observe that
$$
f[A_m]=f[\HH_D(B_{m,0},B_{m,1},\ldots)]=\HH_D(f[B_{m,0}],f[B_{m,1}],\ldots)\in\bG_D(W)
$$
for every $m\in\omega$, where the second equality follows from Proposition \ref{relativizationsettheoretic}.\ref{homeomorphismsettheoretic}.
\end{proof}

\begin{corollary}\label{movexi}
Let $Z$ be an uncountable zero-dimensional Polish space, let $D\subseteq\PP(\omega)$, and let $\xi<\omega_1$. Then $\bG_D(Z)^{(\xi)}=\bG_D^{(\xi)}(Z)$.
\end{corollary}
\begin{proof}
The inclusion $\bG_D(Z)^{(\xi)}\subseteq\bG_D^{(\xi)}(Z)$ follows from Lemma \ref{expansionrelativization}.\ref{expansionmeasurablepreimage}. In order to prove the other inclusion, pick $A\in\bG_D^{(\xi)}(Z)$. By Corollary \ref{expansionbijection}, we can fix a zero-dimensional Polish space $W$ and a $\mathbf{\Sigma}^0_{1+\xi}$-measurable bijection $f:Z\longrightarrow W$ such that $f[A]\in\bG_D(W)$. Since $Z$ contains a copy of $2^\omega$ and $W$ is zero-dimensional, using Lemma \ref{relativization}.\ref{homeomorphism} we can assume without loss of generality that $W$ is a subspace of $Z$, so that $f:Z\longrightarrow Z$. By Lemma \ref{relativization}.\ref{subspace}, we can fix $B\in\bG_D(Z)$ such that $B\cap W=f[A]$. It is easy to check that $A=f^{-1}[B]$, which concludes the proof.
\end{proof}

\begin{corollary}\label{expansionnonselfdual}
Assume $\AD$. Let $Z$ be an uncountable zero-dimensional Polish space, and let $\xi<\omega_1$. Then $\bG^{(\xi)}\in\NSD(Z)$ for every $\bG\in\NSD(Z)$.
\end{corollary}
\begin{proof}
This follows from Corollary \ref{hausdorffevery}, Corollary \ref{movexi}, and Proposition \ref{expansionhausdorffnonselfdual}.
\end{proof}

\begin{corollary}\label{expansionorder}
Assume $\AD$. Let $Z$ be an uncountable zero-dimensional Polish space, and let $\xi<\omega_1$. Then $\bG\subseteq\bL$ iff $\bG^{(\xi)}\subseteq\bL^{(\xi)}$ for every $\bG,\bL\in\NSD(Z)$.
\end{corollary}
\begin{proof}
The fact that $\bG\subseteq\bL$ implies $\bG^{(\xi)}\subseteq\bL^{(\xi)}$ is a trivial consequence of the definition of expansion. Now fix $\bG,\bL\in\NSD(Z)$ such that $\bG^{(\xi)}\subseteq\bL^{(\xi)}$. Assume, in order to get a contradiction, that $\bG\nsubseteq\bL$. Then $\bLc\subseteq\bG$ by Lemma \ref{wadgelemma}, hence
$$
\widecheck{\bL^{(\xi)}}=\bLc^{(\xi)}\subseteq\bG^{(\xi)}\subseteq\bL^{(\xi)}.
$$
Since $\bL^{(\xi)}$ is non-selfdual by Corollary \ref{expansionnonselfdual}, this is a contradiction.
\end{proof}

\section{The expansion theorem}

The main result of this section is Theorem \ref{expansiontheorem}, which will be a crucial tool in obtaining the closure properties in the next section, and will be referred to as the expansion theorem. The proof given here is essentially the same as \cite[proof of Theorem 7.3.9.ii]{louveaub}. This result can be traced back to \cite[Th\'eor\`{e}me 8]{louveausaintraymondf}, which is however limited to the Borel context. We need to introduce the following two notions. The first is \cite[Definition D1]{wadget}, while the second is taken from \cite{louveausaintraymondf} (see also \cite[Section 7.3.4]{louveaub}).\footnote{\,In \cite{louveausaintraymondf}, the notation $\mathbf{\Delta}_{1+\xi}^0\text{-}\PU$ is used instead of $\PU_\xi$, and $\lambda_\mathsf{C}$ is used instead of $\ell$.}

\begin{definition}[Wadge]
Let $Z$ be a space, let $\bG\subseteq\PP(Z)$, and let $\xi<\omega_1$. Define $\PU_\xi(\bG)$ to be the collection of all sets of the form
$$
\bigcup_{n\in\omega}(A_n\cap V_n),
$$
where each $A_n\in\bG$, each $V_n\in\mathbf{\Delta}_{1+\xi}^0(Z)$, the $V_n$ are pairwise disjoint, and $\bigcup_{n\in\omega}V_n=Z$. A set in this form is called a \emph{partitioned union} of sets in $\bG$.
\end{definition}

Notice that the sets $V_n$ in the above definition are not required to be non-empty. It is easy to check that $\PU_\xi(\bG)$ is continuously closed whenever $\bG$ is. Furthermore, it is clear that
$$
\bG\subseteq\PU_\xi(\bG)\subseteq\PU_\eta(\bG)
$$
whenever $\bG\subseteq\PP(Z)$ and $\xi\leq\eta<\omega_1$.

\begin{definition}[Louveau, Saint-Raymond]
Let $Z$ be a space, let $\bG\subseteq\PP(Z)$ be continuously closed, and let $\xi<\omega_1$. Define
\begin{itemize}
\item $\ell(\bG)\geq\xi$ if $\PU_\xi(\bG)=\bG$,
\item $\ell(\bG)=\xi$ if $\ell(\bG)\geq\xi$ and $\ell(\bG)\not\geq\xi+1$,
\item $\ell(\bG)=\omega_1$ if $\ell(\bG)\geq\eta$ for every $\eta<\omega_1$.
\end{itemize}
We refer to $\ell(\bG)$ as the \emph{level} of $\bG$.
\end{definition}

As a trivial example, observe that $\ell(\{\varnothing\})=\ell(\{Z\})=\omega_1$. Using the definition of Wadge-reduction, it is a simple exercise to see that $\ell(\bG)\geq 0$ for every Wadge class $\bG$. We remark that it is not clear at this point whether for every non-selfdual Wadge class $\bG$ there exists $\xi\leq\omega_1$ such that $\ell(\bG)=\xi$.\footnote{\,For example, it is conceivable that $\PU_\eta(\bG)=\bG$ for all $\eta<\xi$, where $\xi$ is a limit ordinal, while $\PU_\xi(\bG)\neq\bG$.} This happens to be true under $\AD$, and it can be proved using techniques from \cite{louveaub} (see \cite[Corollary 18.3]{carroymedinimuller}). However, this fact is never used in this article.

\begin{theorem}\label{expansiontheorem}
Assume $\AD$. Let $Z$ be an uncountable zero-dimensional Polish space, let $\bG\in\NSD(Z)$, and let $\xi<\omega_1$. Then the following conditions are equivalent:
\begin{enumerate}
\item\label{levelgeqxi} $\ell(\bG)\geq\xi$,
\item\label{expansionofsomething} $\bG =\bL^{(\xi)}$ for some $\bL\in\NSD(Z)$.
\end{enumerate}
\end{theorem}
\begin{proof}
In order to show that $(\ref{levelgeqxi})\rightarrow (\ref{expansionofsomething})$, assume that $\ell(\bG)\geq\xi$. Let $\bL\in\NSD(Z)$ be minimal with respect to the property that $\bG\subseteq\bL^{(\xi)}$. Assume, in order to get a contradiction, that $\bL^{(\xi)}\nsubseteq\bG$. It follows from Lemma \ref{wadgelemma} that $\bG\subseteq\bLc^{(\xi)}$, hence $\bG\subseteq\Delta(\bL^{(\xi)})$. Fix $A\subseteq Z$ such that $\bG=[A]$. Also fix $D,E\subseteq\PP(\omega)$ such that $\bG=\bG_D(Z)$ and $\bL=\bG_E(Z)$. Then $\{A,Z\setminus A\}\subseteq\bG_E(Z)^{(\xi)}=\bG_E^{(\xi)}(Z)$, where the equality holds by Corollary \ref{movexi}. Then, by Corollary \ref{expansionbijection}, we can fix a zero-dimensional Polish space $W$ and a $\mathbf{\Sigma}^0_{1+\xi}$-measurable bijection $f:Z\longrightarrow W$ such that $\{f[A],f[Z\setminus A]\}\subseteq\bG_E(W)$.

Next, we will show that $[f[A]]\in\SD(W)$. Assume, in order to get a contradiction, that this is not the case. By Corollary \ref{hausdorffevery}, we can fix $F\subseteq\PP(\omega)$ such that $[f[A]]=\bG_F(W)$. Notice that $\bG_F(W)\subseteq\bG_E(W)$. Furthermore $W\setminus f[A]=f[Z\setminus A]\in\bG_E(W)$, hence $\bGc_F(W)\subseteq\bG_E(W)$. Since $\bG_E(W)$ is non-selfdual by Theorem \ref{addison}, it follows that $\bG_F(W)\subsetneq\bG_E(W)$. Therefore, $\bG_F(Z)\subsetneq\bG_E(Z)=\bL$ by Proposition \ref{orderisomorphism}. On the other hand, Lemma \ref{expansionrelativization}.\ref{expansionmeasurablepreimage} and Corollary \ref{movexi} show that $A=f^{-1}[f[A]]\in\bG_F^{(\xi)}(Z)=\bG_F(Z)^{(\xi)}$. Hence $\bG\subseteq\bG_F(Z)^{(\xi)}$, which contradicts the minimality of $\bL$.

Since $[f[A]]\in\SD(W)$, by Corollaries \ref{selfdualcorollary} and \ref{hausdorffevery}, we can fix $A_n\subseteq W$, $G_n\subseteq\PP(\omega)$ and pairwise disjoint $V_n\in\mathbf{\Delta}^0_1(W)$ for $n\in\omega$ such that $f[A]=\bigcup_{n\in\omega}(A_n\cap V_n)$ and $A_n\in\bG_{G_n}(W)\subsetneq\bG_E(W)$ for each $n$. Notice that $\bG_{G_n}(Z)\subsetneq\bG_E(Z)$ for each $n$ by Proposition \ref{orderisomorphism}, hence $\bG_D(Z)\nsubseteq\bG_{G_n}(Z)^{(\xi)}$ for each $n$ by the minimality of $\bL$. It follows from Corollary \ref{movexi} and Lemma \ref{wadgelemma} that $\bGc_{G_n}^{(\xi)}(Z)\subseteq\bG_D(Z)$. Then, using Propositions \ref{movexidown} and \ref{orderisomorphism}, one sees that $\bGc_{G_n}^{(\xi)}(W)\subseteq\bG_D(W)$.

Set $B_n=W\setminus A_n\in\bGc_{G_n}(W)$ for $n\in\omega$. Observe that $f^{-1}[B_n]\in\bGc_{G_n}^{(\xi)}(Z)\subseteq\bG_D(Z)=\bG$ for each $n$ by Lemma \ref{expansionrelativization}.\ref{expansionmeasurablepreimage}. Furthermore, it is clear that $f^{-1}[V_n]\in\mathbf{\Delta}^0_{1+\xi}(Z)$ for each $n$. In conclusion, since $W\setminus f[A]=\bigcup_{n\in\omega}(B_n\cap V_n)$, we see that
$$
Z\setminus A=\bigcup_{n\in\omega}(f^{-1}[B_n]\cap f^{-1}[V_n])\in\PU_\xi(\bG)=\bG,
$$
where the last equality uses the assumption that $\ell(\bG)\geq\xi$. This contradicts the fact that $\bG$ is non-selfdual.

In order to show that $(\ref{expansionofsomething})\rightarrow(\ref{levelgeqxi})$, let $\bL\in\NSD(Z)$ be such that $\bL^{(\xi)}=\bG$. Pick $A_n\in\bG$ and pairwise disjoint $V_n\in\mathbf{\Delta}^0_{1+\xi}(Z)$ for $n\in\omega$ such that $\bigcup_{n\in\omega}V_n=Z$. We need to show that $\bigcup_{n\in\omega}(A_n\cap V_n)\in\bG$. By Corollary \ref{hausdorffevery}, we can fix $D\subseteq\PP(\omega)$ such that $\bL=\bG_D(Z)$. By Corollary \ref{expansionbijection}, we can fix a Polish space $W$ and a $\mathbf{\Sigma}^0_{1+\xi}$-measurable bijection $f:Z\longrightarrow W$ such that each $f[A_n]\in\bG_D(W)$ and each $f[V_n]\in\mathbf{\Delta}^0_1(W)$. Let $B=\bigcup_{n\in\omega}(f[A_n]\cap f[V_n])$. Since $\bG_D(W)$ is a Wadge class in $W$ by Theorem \ref{addison}, one sees that $B\in\PU_0(\bG_D(W))=\bG_D(W)$. It follows from Lemma \ref{expansionrelativization}.\ref{expansionmeasurablepreimage} that 
$$
\bigcup_{n\in\omega}(A_n\cap V_n)=f^{-1}[B]\in\bG^{(\xi)}_D(Z)=\bL^{(\xi)}=\bG,
$$
where the second equality holds by Corollary \ref{movexi}.
\end{proof}

\begin{corollary}\label{levelinvariance}
Assume $\AD$. Let $Z$ and $W$ be uncountable zero-dimensional Polish spaces, let $D\subseteq\PP(\omega)$, and let $\xi<\omega_1$. Then $\ell(\bG_D(Z))\geq\xi$ iff $\ell(\bG_D(W))\geq\xi$.
\end{corollary}
\begin{proof}
We will only prove the left-to-right implication, as the other one can be proved similarly. Assume that $\ell(\bG_D(Z))\geq\xi$. Then, by Theorems \ref{addison} and \ref{expansiontheorem}, there exists $\bL\in\NSD(Z)$ such that $\bL^{(\xi)}=\bG_D(Z)$. By Corollary \ref{hausdorffevery}, we can fix $E\subseteq\PP(\omega)$ such that $\bL=\bG_E(Z)$. By Proposition \ref{movexidown}, we can fix $F\subseteq\PP(\omega)$ such that $\bG_F(Z)=\bG_E^{(\xi)}(Z)$ and $\bG_F(W)=\bG_E^{(\xi)}(W)$. Notice that $\bG_F(Z)=\bG_D(Z)$ by Corollary \ref{movexi}, hence $\bG_F(W)=\bG_D(W)$ by Proposition \ref{orderisomorphism}. By applying  Corollary \ref{movexi} again, we see that $\bG_D(W)=\bG_E(W)^{(\xi)}$, hence $\ell(\bG_D(W))\geq\xi$ by Theorems \ref{addison} and \ref{expansiontheorem}.
\end{proof}

\section{Good Wadge classes}

The following key notion is essentially due to van Engelen, although he did not give it a name. One important difference is that van Engelen's treatment of this notion is fundamentally tied to Louveau's classification of the Borel Wadge classes from \cite{louveaua}, hence it is limited to the Borel context. The notion of level and the expansion theorem allow us to completely bypass \cite{louveaua}, and extend this concept to arbitrary Wadge classes.

\begin{definition}
Let $Z$ be a space, and let $\bG$ be a Wadge class in $Z$. We will say that $\bG$ is \emph{good} if the following conditions are satisfied:
\begin{itemize}
\item $\bG$ is non-selfdual,
\item $\Delta(\mathsf{D}_\omega(\mathbf{\Sigma}^0_2(Z)))\subseteq\bG$,
\item $\ell(\bG)\geq 1$.
\end{itemize}
\end{definition}

The following proposition gives some concrete examples of good Wadge classes.
\begin{proposition}\label{differencesgood}
Let $Z$ be an uncountable zero-dimensional Polish space, let $\omega\leq\eta<\omega_1$, and let $2\leq\xi<\omega_1$. Then $\mathsf{D}_\eta(\mathbf{\Sigma}^0_\xi(Z))$ is a good Wadge class in $Z$.
\end{proposition}
\begin{proof}
Set $\bG=\mathsf{D}_\eta(\mathbf{\Sigma}^0_\xi(Z))$. The fact that $\bG\in\NSD(Z)$ follows from Propositions \ref{expansiondifferences} and \ref{expansionhausdorffnonselfdual}. The inclusion $\Delta(\mathsf{D}_\omega(\mathbf{\Sigma}^0_2(Z)))\subseteq\bG$ holds trivially. Finally, using Corollary \ref{movexi} and \cite[Th\'eor\`{e}me 8]{louveausaintraymondf},\footnote{\,Here, we apply \cite[Th\'eor\`{e}me 8]{louveausaintraymondf} instead of Theorem \ref{expansiontheorem} simply because the former does not require $\AD$.} one sees that $\ell(\bG)\geq 1$.
\end{proof}

The main result of this section is Theorem \ref{closureGdelta}, which will be crucial in showing that good Wadge classes are reasonably closed (see Lemma \ref{goodimpliesreasonablyclosed}). The case $Z=\omega^\omega$ of the following lemma is due to Andretta, Hjorth, and Neeman (see \cite[Lemma 3.6.a]{andrettahjorthneeman}), and the general case follows easily from this particular case (thanks to the machinery of relativization).

\begin{lemma}\label{closureclosed}
Assume $\AD$. Let $Z$ be an uncountable zero-dimensional Polish space, and let $\bG\in\NSD(Z)$. Assume that $\mathsf{D}_n(\mathbf{\Sigma}^0_1(Z))\subseteq\bG$ for every $n\in\omega$.
	\begin{itemize}
	\item If $A\in\bG$ and $C\in\mathbf{\Pi}^0_1(Z)$ then $A\cap C\in\bG$.
	\item If $A\in\bG$ and $U\in\mathbf{\Sigma}^0_1(Z)$ then $A\cup U\in\bG$.
	\end{itemize}
\end{lemma}
\begin{proof}
Observe that, since $\bGc$ also satisfies the assumptions of the lemma, it will be enough to prove the first statement. So pick $A\in\bG$ and $C\in\mathbf{\Pi}^0_1(Z)$. By Corollary \ref{hausdorffevery}, we can fix $D\subseteq\PP(\omega)$ such that $\bG=\bG_D(Z)$. Set $\bL=\bG_D(\omega^\omega)$. Using Lemma \ref{relativization}.\ref{homeomorphism}, we can assume without loss of generality that $Z$ is a closed subspace of $\omega^\omega$. By \cite[Proposition 2.8]{kechris}, we can fix a retraction $\rho:\omega^\omega\longrightarrow Z$. Notice that $\rho^{-1}[A]\in\bL$ by Lemma \ref{relativization}.\ref{preimage}. Furthermore, it is clear that $C\in\mathbf{\Pi}^0_1(\omega^\omega)$.

Next, we claim that $||\bL||\geq\omega$ (see Definition \ref{wadgerank}). Since $\mathsf{D}_n(\mathbf{\Sigma}^0_1(Z))\subseteq\bG$ for every $n\in\omega$, using Propositions \ref{hausdorffdifferences} and \ref{orderisomorphism} one sees that $\mathsf{D}_n(\mathbf{\Sigma}^0_1(\omega^\omega))\subseteq\bL$ for every $n\in\omega$. Since these are Wadge classes by Theorem \ref{addison}, and they form a strictly increasing sequence by \cite[Exercise 22.26.iv]{kechris}, our claim is proved. Therefore, we can apply \cite[Lemma 3.6.a]{andrettahjorthneeman}, which shows that $\rho^{-1}[A]\cap C\in\bL$. Finally, Lemma \ref{relativization}.\ref{subspace} shows that $A\cap C=(\rho^{-1}[A]\cap C)\cap Z\in\bG_D(Z)=\bG$.
\end{proof}

\begin{theorem}\label{closureGdelta}
Assume $\AD$. Let $Z$ be an uncountable zero-dimensional Polish space, and let $\bG\in\NSD(Z)$. Assume that $\mathsf{D}_n(\mathbf{\Sigma}^0_2(Z))\subseteq\bG$ for every $n\in\omega$ and $\ell(\bG)\geq 1$.
\begin{itemize}
	\item If $A\in\bG$ and $G\in\mathbf{\Pi}^0_2(Z)$ then $A\cap G\in\bG$.
	\item If $A\in\bG$ and $F\in\mathbf{\Sigma}^0_2(Z)$ then $A\cup F\in\bG$.
	\end{itemize}
In particular, the above two statements hold for every good Wadge class $\bG$ in $Z$.
\end{theorem}
\begin{proof}
Observe that, since $\bGc$ also satisfies the assumptions of the theorem, it will be enough to prove the first statement. So pick $A\in\bG$ and $G\in\mathbf{\Pi}^0_2(Z)$. By Theorem \ref{expansiontheorem}, we can pick $\bL\in\NSD(Z)$ such that $\bL^{(1)}=\bG$. By Corollary \ref{hausdorffevery}, we can fix $E\subseteq\PP(\omega)$ such that $\bG_E(Z)=\bL$.

Since $\bL^{(1)}=\bG$, there exists a $\mathbf{\Sigma}^0_2$-measurable function $f:Z\longrightarrow Z$ and $B\in\bL$ such that $A=f^{-1}[B]$. Furthermore, using Corollary \ref{movexi} for a suitable choice of $D$, it is easy to check that $\mathbf{\Pi}^0_1(Z)^{(1)}=\mathbf{\Pi}^0_2(Z)$. Therefore, there exists a $\mathbf{\Sigma}^0_2$-measurable function $g:Z\longrightarrow Z$ and $C\in\mathbf{\Pi}^0_1(Z)$ such that $G=g^{-1}[C]$. By applying Lemma \ref{relativization}.\ref{preimage} to the projection on the first coordinate $\pi:Z\times Z\longrightarrow Z$, one sees that $B\times Z\in\bG_E(Z\times Z)$. Furthermore, it is clear that $Z\times C\in\mathbf{\Pi}^0_1(Z\times Z)$.

We claim that $\mathsf{D}_n(\mathbf{\Sigma}^0_1(Z\times Z))
\subseteq\bG_E(Z\times Z)$ for every $n\in\omega$. So fix $n\in\omega$, and let $D\subseteq\PP(\omega)$ be the set given by Proposition \ref{expansiondifferences} when $\eta=n$. Notice that
$$
\bG_D(Z)^{(1)}=\bG_D^{(1)}(Z)=\mathsf{D}_n(\mathbf{\Sigma}^0_2(Z))\subseteq\bG=\bG_E(Z)^{(1)},
$$
where the first equality holds by Corollary \ref{movexi}. Therefore $\bG_D(Z)\subseteq\bG_E(Z)$ by Corollary \ref{expansionorder}. An application of Proposition \ref{orderisomorphism} with $W=Z\times Z$ concludes the proof of our claim.

Therefore, we can apply Lemma \ref{closureclosed}, which shows that $B\times C=(B\times Z)\cap (Z\times C)\in\bG_E(Z\times Z)$. Consider the function $(f,g):Z\longrightarrow Z\times Z$ defined by $(f,g)(x)=(f(x),g(x))$, and observe that $(f,g)$ is $\mathbf{\Sigma}^0_2$-measurable. By Lemma \ref{expansionrelativization}.\ref{expansionmeasurablepreimage}, it follows that
$$
A\cap G=(f,g)^{-1}[B\times C]\in\bG_E^{(1)}(Z)=\bL^{(1)}=\bG,
$$
where the second equality holds by Corollary \ref{movexi}.
\end{proof}

\section{Reasonably closed Wadge classes}

In this section we will define reasonably closed Wadge classes and prove that every good Wadge class is reasonably closed. This notion is an ad hoc definition, and it is the key idea of an ingenious lemma due to Harrington (see \cite[Lemma 3]{steela}). This lemma is a crucial ingredient in the proof of Theorem \ref{steeltheorem}. Here, we will follow the approach of \cite[Section 4.1]{vanengelent}.

Given $i\in 2$, set
$$
Q_i=\{x\in 2^\omega:x(n)=i\text{ for all but finitely many }n\in\omega\}.
$$
Notice that every element of $2^\omega\setminus(Q_0\cup Q_1)$ is obtained by alternating finite blocks of zeros and finite blocks of ones. Define the function $\phi:2^\omega\setminus(Q_0\cup Q_1)\longrightarrow 2^\omega$ by setting
$$
\phi(x)(n)=\left\{
\begin{array}{ll} 0 & \textrm{if the $n^{\text{th}}$ block of zeros of $x$ has even length},\\
1 & \textrm{otherwise,}
\end{array}
\right.
$$
where we start counting with the $0^{\text{th}}$ block of zeros. It is easy to check that $\phi$ is continuous.
\begin{definition}
Let $\bG$ be a Wadge class in $2^\omega$. We will say that $\bG$ is \emph{reasonably closed} if $\phi^{-1}[A]\cup Q_0\in\bG$ for every $A\in\bG$.	
\end{definition}

The following result is essentially the same as \cite[Lemma 4.2.17]{vanengelent}, except that it is not limited to the Borel context.

\begin{lemma}\label{goodimpliesreasonablyclosed}
Assume $\AD$. Let $\bG$ be a good Wadge class in $2^\omega$. Then $\bG$ is reasonably closed.
\end{lemma}
\begin{proof}
By Corollary \ref{hausdorffevery}, we can fix $D\subseteq\PP(\omega)$ such that $\bG=\bG_D(2^\omega)$. Set $Z=2^\omega\setminus(Q_0\cup Q_1)$. Pick $A\in\bG$. Notice that $\phi^{-1}[A]\in\bG_D(Z)$ by Lemma \ref{relativization}.\ref{preimage}. Therefore, by Lemma \ref{relativization}.\ref{subspace}, there exists $B\in\bG$ such that $B\cap Z=\phi^{-1}[A]$. Since $\bG$ is a good Wadge class and $Z\in\mathbf{\Pi}^0_2(2^\omega)$, it follows from Theorem \ref{closureGdelta} that $\phi^{-1}[A]\in\bG$. Finally, again by Theorem \ref{closureGdelta}, we see that $\phi^{-1}[A]\cup Q_0\in\bG$, which concludes the proof.
\end{proof}

\section{Wadge classes of homogeneous spaces are good}

The main result of this section is that $[X]$ is a good Wadge class whenever $X$ is a homogeneous space of sufficiently high complexity (see Theorem \ref{classesofhomogeneousspaces} for the precise statement). Together with Lemma \ref{goodimpliesreasonablyclosed}, this will allow us to apply Theorem \ref{steeltheorem} in the next section.

We will need three preliminary results. Lemmas \ref{closedunderhomeomorphisms}, \ref{openimplieswhole}, and \ref{l0impliessomewheresmaller} correspond to  \cite[Lemma 4.2.16]{vanengelent}, \cite[Lemma 4.4.2]{vanengelent}, and \cite[Lemma 4.4.1]{vanengelent} respectively, while Theorem \ref{classesofhomogeneousspaces} corresponds to \cite[Lemma 4.4.3]{vanengelent}. Once again, the difference is that we work with arbitrary sets instead of just Borel sets. In the case of Lemma \ref{l0impliessomewheresmaller}, this yields at the same time a substantially simpler proof, inspired by \cite[proof of Theorem 7.3.10.ii]{louveaub}.

\begin{lemma}\label{closedunderhomeomorphisms}
Assume $\AD$. Let $Z$ be an uncountable zero-dimensional Polish space, and let $\bG$ be a good Wadge class in $Z$. Assume that $A$ and $B$ are subspaces of $Z$ such that $B\in\bG$ and $A\approx B$. Then $A\in\bG$.
\end{lemma}
\begin{proof}
Let $h:A\longrightarrow B$ be a homeomorphism. By \cite[Theorem 3.9]{kechris}, we can fix $G,H\in\mathbf{\Pi}^0_2(Z)$ and a homeomorphism $f:G\longrightarrow H$ such that $h\subseteq f$. By Corollary \ref{hausdorffevery}, we can fix $D\subseteq\PP(\omega)$ such that $\bG=\bG_D(Z)$. Notice that $B\in\bG_D(H)$ by Lemma \ref{relativization}.\ref{subspace}. It follows from Lemma \ref{relativization}.\ref{preimage} that $A\in\bG_D(G)$. Therefore, according to Lemma \ref{relativization}.\ref{subspace}, there exists $C\in\bG_D(Z)$ such that $C\cap G=A$. Since $G\in\mathbf{\Pi}^0_2(Z)$, an application of Theorem \ref{closureGdelta} concludes the proof.
\end{proof}

\begin{lemma}\label{openimplieswhole}
Assume $\AD$. Let $Z$ be an uncountable zero-dimensional Polish space, let $\bG$ be a good Wadge class in $Z$, and let $X$ be a homogeneous subspace of $Z$. Assume that $A\in\mathbf{\Sigma}^0_1(X)$ is non-empty and $A\in\bG$. Then $X\in\bG$.
\end{lemma}
\begin{proof}
Define $\UU=\{h[A]:h\text{ is a homeomorphism of }X\}$. Notice that $\UU$ is a cover of $X$ because $X$ is homogeneous and $A$ is non-empty. Let $\{A_n:n\in\omega\}$ be a countable subcover of $\UU$. Observe that each $A_n\in\bG$ by Lemma \ref{closedunderhomeomorphisms}. Fix $U_n\in\mathbf{\Sigma}^0_1(Z)$ for $n\in\omega$ such that $U_n\cap X=A_n$ for each $n$. Set $V_n=U_n\setminus\bigcup_{k<n}U_k$ for $n\in\omega$, and observe that  $V_n\in\mathbf{\Delta}^0_2(Z)$ for each $n$. Furthermore, it is easy to check that
$$
X=\bigcup_{-1\leq n<\omega}(V_n\cap A_n),
$$
where $V_{-1}=Z\setminus\bigcup_{n<\omega}V_n=Z\setminus\bigcup_{n<\omega}U_n$ and $A_{-1}=\varnothing$. In conclusion, we see that $X\in\PU_1(\bG)$. Since $\ell(\bG)\geq 1$, it follows that $X\in\bG$.
\end{proof}

\begin{lemma}\label{l0impliessomewheresmaller}
Assume $\AD$. Let $Z$ be an uncountable zero-dimensional Polish space, let $\bG\in\NSD(Z)$ be such that $\ell(\bG)=0$, and let $X\in\bG$ be codense in $Z$. Then there exists a non-empty $U\in\mathbf{\Delta}^0_1(Z)$ and $\bL\in\NSD(Z)$ such that $\bL\subsetneq\bG$ and $X\cap U\in\bL$.	
\end{lemma}
\begin{proof}
Since $\ell(\bG)=0$, using Lemma \ref{wadgelemma} and the fact that $\mathsf{PU}_1(\bG)$ is continuously closed, it is easy to see that $\bGc\subseteq\mathsf{PU}_1(\bG)$. Therefore, we can fix $A_n\in\bG$ and pairwise disjoint $V_n\in\mathbf{\Delta}^0_2(Z)$ for $n\in\omega$ such that $\bigcup_{n\in\omega}(A_n\cap V_n)=Z\setminus X$. Since $Z$ is a Baire space, we can fix $n\in\omega$ and a non-empty $U\in\mathbf{\Delta}^0_1(Z)$ such that $U\subseteq V_n$.

Notice that $\bG\neq\{Z\}$ and $\bG\neq\{\varnothing\}$ because $\ell(\bG)=0$, hence it is possible to apply Proposition \ref{closureclopen}. In particular, one sees that $U\setminus X=A_n\cap U\in\bG$, hence $Z\setminus(X\cap U)=(Z\setminus U)\cup (U\setminus X)\in\bG$. So, we have $X\cap U\in\bG$ (again by Proposition \ref{closureclopen}) and $Z\setminus(X\cap U)\in\bG$. This easily yields the desired result if $X\cap U$ is non-selfdual, so assume that $X\cap U$ is selfdual. By Corollary \ref{selfdualcorollary}, we can fix pairwise disjoint $U_n\in\mathbf{\Delta}^0_1(U)$ and non-selfdual $B_n<X\cap U$ in $Z$ for $n\in\omega$ such that $\bigcup_{n\in\omega}U_n=U$ and $\bigcup_{n\in\omega}(B_n\cap U_n)=X\cap U$. If we had $B_n=Z$ for each $n$ such that $U_n\neq\varnothing$ then the assumption that $X$ is codense in $Z$ would be contradicted, so assume that $n\in\omega$ is such that $B_n\neq Z$ and $U_n\neq\varnothing$. To conclude the proof, set $\bL=[B_n]$ and observe that $X\cap U_n=B_n\cap U_n\leq B_n$ by Proposition \ref{closureclopen}.
\end{proof}

\begin{theorem}\label{classesofhomogeneousspaces}
Assume $\AD$. Let $Z$ be an uncountable zero-dimensional Polish space, and let $X$ be a homogeneous dense subspace of $Z$ such that $X\notin\Delta(\mathsf{D}_\omega(\mathbf{\Sigma}^0_2(Z)))$. Then $[X]$ is a good Wadge class in $Z$.
\end{theorem}
\begin{proof}
Fix $\bG\in\NSD(Z)$ minimal with respect to the property that $X\cap U\in\bG\cup\bGc$ for some non-empty $U\in\mathbf{\Delta}^0_1(Z)$. Fix a non-empty $U\in\mathbf{\Delta}^0_1(Z)$ such that $X\cap U\in\bG\cup\bGc$. Assume without loss of generality that $X\cap U\in\bG$ (the case $X\cap U\in\bGc$ is similar). First we will prove that $\bG$ is a good Wadge class, then we will show that $[X]=\bG$. Observe that $\mathsf{D}_\omega(\mathbf{\Sigma}^0_2(Z))$ and $\check{\mathsf{D}}_\omega(\mathbf{\Sigma}^0_2(Z))$ are good Wadge classes in $Z$ by Proposition \ref{differencesgood}. We claim that $X\cap U\notin\Delta(\mathsf{D}_\omega(\mathbf{\Sigma}^0_2(Z)))$. Assume, in order to get a contradiction, that $X\cap U\in\Delta(\mathsf{D}_\omega(\mathbf{\Sigma}^0_2(Z)))$. Then, by the density of $X$, it is possible to apply Lemma \ref{openimplieswhole} (twice), obtaining that $X\in\Delta(\mathsf{D}_\omega(\mathbf{\Sigma}^0_2(Z)))$. Since this contradicts our assumptions, our claim is proved. By Lemma \ref{wadgelemma}, it follows that $\Delta(\mathsf{D}_\omega(\mathbf{\Sigma}^0_2(Z)))\subseteq\bG$.

Next, we claim that $\ell(\bG)\geq 1$. Assume, in order to get a contradiction, that $\ell(\bG)=0$. By Corollary \ref{hausdorffevery}, we can fix $D\subseteq\PP(\omega)$ such that $\bG_D(Z)=\bG$. Since $X$ is dense in $Z$ and homogeneous, if $U$ were countable then $X$ would be countable, by the same argument as in the proof of Proposition \ref{somewherepolish}. So $U$ is an uncountable zero-dimensional Polish space, and $\ell(\bG_D(U))=0$ by Corollary \ref{levelinvariance}. Furthermore, $X$ must be codense in $Z$, otherwise it would follow that $X$ is Polish by Proposition \ref{somewherepolish}, hence $X\in\mathbf{\Pi}^0_2(Z)$ by \cite[Theorem 3.11]{kechris}. Therefore, by Lemma \ref{l0impliessomewheresmaller}, there exists a non-empty $V\in\mathbf{\Delta}^0_1(U)$ and $\bL\in\NSD(U)$ such that $\bL\subsetneq\bG_D(U)$ and $X\cap V\in\bL$. By Corollary \ref{hausdorffevery}, we can fix $E\subseteq\PP(\omega)$ such that $\bG_E(U)=\bL$. Observe that $\bG_E(Z)\subsetneq\bG_D(Z)$ by Proposition \ref{orderisomorphism}. Therefore, in order to contradict the minimality of $\bG$, it remains to show that $X\cap V\in\bG_E(Z)$. By Lemma \ref{relativization}.\ref{subspace}, there exists $A\in\bG_E(Z)$ such that $A\cap U=X\cap V$. Notice that $\bG_E(Z)\neq\{Z\}$, otherwise it would follow that $X=Z$, which contradicts our assumptions. Therefore $X\cap V=A\cap U\in\bG_E(Z)$ by Proposition \ref{closureclopen}.

At this point, we know that $\bG$ is a good Wadge class, so we can apply Lemma \ref{openimplieswhole}, obtaining that $X\in\bG$. To conclude the proof, it will be enough to show that $X$ is non-selfdual, as it will follow from the minimality of $\bG$ and Proposition \ref{closureclopen} that $[X]=\bG$. Assume, in order to get a contradiction, that $X$ is selfdual. Then, by Corollary \ref{selfdualcorollary}, there exist a non-empty $V\in\mathbf{\Delta}^0_1(Z)$ and a non-selfdual $A < X$ in $Z$ such that $A\cap V=X\cap V$. Set $\bL=[A]$, and observe that $\bL\subsetneq\bG$. Notice that $\bL\neq\{Z\}$, otherwise it would follow that $V\subseteq X$, hence $X$ would not be codense in $Z$. Therefore $X\cap V=A\cap V\in\bL$ by Proposition \ref{closureclopen}. This contradicts the minimality of $\bG$.
\end{proof}

\section{The main results}

This sections contains our main results. Theorem \ref{uniqueness} extends (and is inspired by) \cite[Lemma 2.7]{vanengelena}. All the work done so far was aimed at applying the following result, which is a particular case of \cite[Theorem 2]{steela}. Given a Wadge class $\bG$ in $2^\omega$ and $X\subseteq 2^\omega$, we will say that $X$ is \emph{everywhere properly $\bG$} if $X\cap [s]\in\bG\setminus\bGc$ for every $s\in 2^{<\omega}$.

\begin{theorem}[Steel]\label{steeltheorem}
Assume $\AD$. Let $\bG$ be a reasonably closed Wadge class in $2^\omega$. Assume that $X$ and $Y$ are subsets of $2^\omega$ that satisfy the following conditions:
\begin{itemize}
\item $X$ and $Y$ are everywhere properly $\bG$,
\item $X$ and $Y$ are either both meager in $2^\omega$ or both comeager in $2^\omega$.
\end{itemize}
Then there exists a homeomorphism $h:2^\omega\longrightarrow 2^\omega$ such that $h[X]=Y$.
\end{theorem}

\begin{theorem}\label{uniqueness}
Assume $\AD$. Let $X$ and $Y$ be homogeneous dense subspaces of $2^\omega$. Assume that $X\notin\Delta(\mathsf{D}_\omega(\mathbf{\Sigma}^0_2(2^\omega)))$, and that the following conditions are satisfied:
\begin{itemize}
\item $[X]=[Y]$,
\item $X$ and $Y$ are either both meager spaces or both Baire spaces.
\end{itemize}
Then there exists a homeomorphism $h:2^\omega\longrightarrow 2^\omega$ such that $h[X]=Y$.
\end{theorem}
\begin{proof}
Let $\bG=[X]$. Notice that $\bG$ is a good Wadge class by Theorem \ref{classesofhomogeneousspaces}, hence it is reasonably closed by Lemma \ref{goodimpliesreasonablyclosed}. It is clear that if $X$ and $Y$ are both meager spaces, then they are both meager in $2^\omega$. On the other hand, if $X$ and $Y$ are both Baire spaces, then they are comeager in $2^\omega$ by Proposition \ref{bairecomeager}. Hence, by Theorem \ref{steeltheorem}, it will be enough to show that $X$ and $Y$ are everywhere properly $\bG$. We will only prove this for $X$, since the proof for $Y$ is perfectly analogous. Pick $s\in 2^{<\omega}$. Using Proposition \ref{closureclopen}, one sees that $X\cap [s]\in\bG$. In order to get a contradiction, assume that $X\cap [s]\in\bGc$. Since $\bGc$ is also a good Wadge class, it follows from Lemma \ref{openimplieswhole} that $X\in\bGc$, which contradicts the fact that $\bG$ is non-selfdual.
\end{proof}

\begin{corollary}\label{main}
Assume $\AD$. Let $X$ be a zero-dimensional homogeneous space that is not locally compact. Then $X$ is strongly homogeneous.	
\end{corollary}
\begin{proof}
Notice that $X$ is crowded, otherwise it would be discrete by homogeneity. Therefore, we can assume without loss of generality that $X$ is a dense subspace of $2^\omega$. If $X\in\Delta(\mathsf{D}_\omega(\mathbf{\Sigma}^0_2(2^\omega)))$, then the desired result follows from \cite[Corollary 4.4.6]{vanengelent}. So assume that $X\notin\Delta(\mathsf{D}_\omega(\mathbf{\Sigma}^0_2(2^\omega)))$.

By Theorem \ref{pibase}, it will be enough to show that $X\cap [s]\approx X$ for every $s\in 2^{<\omega}$. Pick $s\in 2^{<\omega}$. Let $h:[s]\longrightarrow 2^\omega$ be a homeomorphism, and let $Y=h[X\cap [s]]$. It is easy to check that $Y$ is a homogeneous dense subspace of $2^\omega$. Furthermore, it is clear that $X$ and $Y$ are either both meager spaces or both Baire spaces. We claim that $[X]=[Y]$. By Theorem \ref{uniqueness}, this will conclude the proof.

Set $\bG=[X]$, and observe that $\bG$ is a good Wadge class by Theorem \ref{classesofhomogeneousspaces}. In particular, $\bG$ is non-selfdual. Hence, by Corollary \ref{hausdorffevery}, we can fix $D\subseteq\PP(\omega)$ such that $\bG=\bG_D(2^\omega)$. Notice that $X\cap [s]\in\bG_D([s])$ by Lemma \ref{relativization}.\ref{subspace}, hence $Y\in\bG_D(2^\omega)$ by Lemma \ref{relativization}.\ref{homeomorphism}. This shows that $[Y]\subseteq [X]$. In order to prove the other inclusion, by Lemma \ref{wadgelemma}, it will be enough to show that $Y\notin\bGc_D(2^\omega)$. Assume, in order to get a contradiction, that $Y\in\bGc_D(2^\omega)$. Then $X\cap [s]\in\bGc_D([s])$ by Lemma \ref{relativization}.\ref{homeomorphism}. It follows easily from Lemma \ref{relativization}.\ref{subspace} and Proposition \ref{closureclopen} that $X\cap [s]\in\bGc_D(2^\omega)=\bGc$. This implies that $X\in\bGc$ by Lemma \ref{openimplieswhole}, which contradicts the fact that $\bG$ is non-selfdual.
\end{proof}

\end{document}